\documentclass{amsart}
\usepackage[all]{xy}
\usepackage{amsmath,amsfonts,amssymb,amsthm,epsfig,amscd,comment,latexsym,psfrag}
\usepackage{nicefrac,xspace,tikz}
\usetikzlibrary{arrows}

\title{A categorification of the positive half of quantum $gl(m|1)$}
\author{Mikhail Khovanov}
\email{khovanov@math.columbia.edu}
\address{Department of Mathematics \\ Columbia University \\ New York, United States}
\author{Joshua Sussan}
\email{jsussan@mec.cuny.edu}
\address{Department of Mathematics\\ CUNY Medgar Evers \\ New York, United States}
\date{June 6, 2014}

\newtheorem{defined}{Definition}
\newtheorem{prop}{Proposition}
\newtheorem{theorem}{Theorem}
\newtheorem{corollary}{Corollary} 
\newtheorem{lemma}{Lemma}
\newtheorem{remark}{Remark}
\newtheorem*{thm}{Theorem}
\newcommand{\oplusop}[1]{{\mathop{\oplus}\limits_{#1}}}
 \newcommand{\oplusoop}[2]{{\mathop{\oplus}\limits_{#1}^{#2}}}
 
\begin{document} 

\maketitle
\baselineskip 14pt
 
\def\R{\mathbb R}
\def\Q{\mathbb Q}
\def\Z{\mathbb Z}
\def\N{\mathbb N} 
\def\C{\mathbb C}
\def\l{\lbrace}
\def\r{\rbrace}
\def\o{\otimes}
\def\lra{\longrightarrow}
\def\Hom{\mathrm{Hom}}
\def\HOM{\mathrm{HOM}}
\def\RHom{\mathrm{RHom}}
\def\Id{\mathrm{Id}}
\def\mc{\mathcal}
\def\mf{\mathfrak} 
\def\Ext{\mathrm{Ext}}
\def\Ind{\mathrm{Ind}}
\def\Res{\mathrm{Res}}
\def\soc{\mathrm{soc}}
\def\hd{\mathrm{hd}}
\def\Seq{\mathrm{Seq}}
\def\shuffle{\,\raise 1pt\hbox{$\scriptscriptstyle\cup{\mskip
               -4mu}\cup$}\,}
\newcommand{\define}{\stackrel{\mbox{\scriptsize{def}}}{=}}


%
\def\drawing#1{\begin{center}\epsfig{file=#1}\end{center}}

 \def\yesnocases#1#2#3#4{\left\{
\begin{array}{ll} #1 & #2 \\ #3 & #4
\end{array} \right. }

\newcommand{\LOT}{H^-} 

\begin{abstract} 
We describe a collection of differential graded rings that categorify weight 
spaces of the positive half of the quantized universal enveloping algebra 
of the Lie superalgebra $gl(m|1)$. 
\end{abstract}

\tableofcontents 


\section{Introduction}
A categorification of $ \mathcal{U}_q^+(\mathfrak{g}) $, the positive half of the quantum enveloping algebra associated to a simply-laced finite-dimensional simple Lie algebra over $ \C $, was constructed in \cite{KL1}, and independently in \cite{Ro}, using a diagramatically defined algebra $ \oplus_{\nu} R(\nu) $ where $ \nu $ varies over the positive root lattice of $ \mathfrak{g}$.  It was shown that the split Grothendieck group of the category of projective objects of $ \oplus_{\nu} R(\nu) $ is isomorphic as a twisted bialgebra to $ \mathcal{U}_q^+(\mathfrak{g}) $.
A categorification of the enveloping algebra of $ \hat{\mathfrak{sl}}_n $ was realized years earlier by Ariki and Grojnowski \cite{Ari, Groj} by considering the Grothendieck group of the category of representations of the affine Hecke algebra of type $ A$ when the deformation parameter is a root of unity. The techniques of \cite{KL1} are a graded analogue of the work of Grojnowksi and Vazirani ~\cite{GrojVaz}.  See \cite{Klesh} for an exposition of this subject.  

The subject of categorification of quantum super algebras was initiated in ~\cite{Kh1}.  A diagramatically defined algebra was introduced which comes equipped with two gradings.  One of the gradings, not present in the algebras defined in \cite{KL1}, endows this algebra with the structure of a differential graded algebra.  It was shown that the Grothendieck group of a suitably chosen category of dg modules categorifies 
$ \mathcal{U}_q^+(\mathfrak{gl}(2|1)) $.  In the present work we synthesize the algebras of ~\cite{KL1} and ~\cite{Kh1} to produce a new bigraded algebra $ \oplus_{\nu} R_{}(\nu) $, (where $ \nu $ varies over the positive root lattice of $ \mathfrak{gl}(m|n)$), which specializes to a special case of the algebras introduced in ~\cite{KL1} and to the algebra of ~\cite{Kh1}. One of its gradings gives this algebra a differential structure. Induction and restriction functors endow the Grothendieck group (of a suitably chosen category) $ K_0(\oplus_{\nu} R_{}(\nu)) $ with the structure of a twisted bialgebra.  Section ~\ref{dgalgebras} lays the groundwork for the definition of $ K_0(\oplus_{\nu} R_{}(\nu))$.
Our main result is:
\begin{thm}
There exists an isomorphism of twisted bialgebras
\begin{equation*}
\gamma \colon \mathcal{U}_q^+(\mathfrak{gl}(m|1)) \rightarrow \mathbb{Q}(q) \otimes_{\mathbb{Z}[q,q^{-1}]} K_0(\oplus_{\nu} R_{}(\nu)).
\end{equation*}
\end{thm}

\subsection{Other work on categorified superalgebras}
Hill and Wang also studied categorified superalgebras ~\cite{HW}. The superalgebras they consider do not include $ \mathcal{U}_q^+(\mathfrak{gl}(m|1)) $ and the machinery of dg algebras is not necessary in their work.

Clark, Hill, and Wang ~\cite{CHW} constructed a canonical basis for a large class of quantum superalgebras which includes $ \mathcal{U}_q^+(\mathfrak{gl}(1|n)) $.  It would be interesting to compare this basis with the basis obtained from indecomposable projective objects considered here.

Tian categorified a form of quantum $ \mathfrak{gl}(1|1) $ and tensor products of its irreducible representations using dg categories coming from contact geometry ~\cite{Tian1, Tian2}.

A Lie theoretic approach towards categorifying tensor products of quantum $ \mathfrak{gl}(1|1)$-modules could be found in ~\cite{Sa}.  On certain derived subcategories of $ \mathcal{O} $ for $ \mathfrak{gl}(n) $, Sartori constructs functors $ \mathcal{E} $ and $ \mathcal{F} $ whose squares are identically zero.  Sartori's construction seems different from ~\cite{Tian1, Tian2, Kh1} and this paper since he gets categorified odd generators squaring to zero without
having a dg structure present.

\subsection{Outline}
In Section ~\ref{bilinearform} we review the definition of $ \mathcal{U}_q^+(\mathfrak{gl}(m|n)) $ in the style of ~\cite{Lus1} and give a graphical interpretation of the natural bilinear form.   This presentation of the algebra defined as the quotient of a free algebra module that radical of a natural bilinear form is
contained in ~\cite{Yam}.
In Section ~\ref{dgalgebras} some technical issues concerning the Grothendieck groups needed in this work are explained.
The algebra $ R_{}(\nu) $ is defined in Section ~\ref{R(nu)}.  It is shown that the algebra is non-degenerate by constructing a faithful representation on some analogue of the polynomial representation.
Finally in Section ~\ref{maintheorem}, the proof that $ \oplus_{\nu} R_{dg}(\nu) $ categorifies $ \mathcal{U}_q^+(\mathfrak{gl}(m|1)) $ is given.  The general results of Section ~\ref{dgalgebras} as well as a modification of ~\cite[Chapter 5]{Klesh} are used.

\subsection{Acknowledgements}
M.K. was supported by NSF grants DMS-1005750 and DMS-0739392.
J.S. was supported by NSF grant DMS-1407394 and PSC-CUNY Award 67144-00 45.

\section{Superalgebra ${\bf f}(m,n) $ and its bilinear form} 
\label{bilinearform}
We introduce a free twisted super bialgebra $ {\bf f}' $ along with a bilinear form in the style of ~\cite{Lus1}.  Much of this section is devoted to explicitly computing the radical $ \mathcal{I} $ of this form and finding 
a PBW basis for the quotient algebra $ {\bf f} = {\bf f}' / \mathcal{I}$.  We note that Yamane ~\cite{Yam} has proven this result in greater generality, but since our notation differs from his, we write down the proofs again. At the end of the section, we give a graphical interpretation of the bilinear form to be used later as in ~\cite{KL1}.
\subsection{Superalgebras $ \mathfrak{gl}(m | n)$ and $ {\bf f}'(m,n)$.}
Let $ V = \C^{m | n} $ be the $ \Z /2 $- graded vector space where the degree zero part is $ V_0 = \C^m $ with basis $ \lbrace v_1, \ldots, v_m \rbrace $ and the degree one part is $ V_1 = \C^n$ with basis 
$ \lbrace v_{m+1}, \ldots, v_{m+n} \rbrace $.  Let the basis of the dual space $ V_0^* $ be $ \lbrace v_1^*, \ldots, v_m^* \rbrace $ and the basis of $ V_1^* $ be $ \lbrace v_{m+1}^*, \ldots, v_{m+n}^* \rbrace. $
There is a natural bilinear form $ \langle , \rangle $ on $ V^* $ given on the basis elements by 
$$ \langle v_i^*, v_j^* \rangle = \begin{cases}
\delta_{i,j} & \textrm{ if } v_i^* \in V_0^* \\
-\delta_{i,j} & \textrm{ if } v_i^* \in V_1^*.
\end{cases} $$
Let $ \mathfrak{gl}(m | n) $ be the Lie superalgebra of super endomorphisms of $ V$. 
Note that the form $ \langle , \rangle $ above induces a form on the dual of the Cartan subalgebra of $ \mathfrak{gl}(m | n) $.

Let $ I = \l 1, \ldots, m+n-1 \r $, $ I' = \l 1, \ldots, m-1\r $, $ I'' = \l m+1, \ldots, m+n-1 \r$ and $ \nu = \ \sum_i \nu_i i \in \N[I]$.

Let $ \Delta^+ $ be the set of positive roots for $ \mathfrak{gl}(m | n)$.
An element of this set is of the form $ \alpha_i + \cdots + \alpha_j $ where $ i \leq j$ and $ i,j \in I$. 
Define a function $ p \colon \Delta^+ \rightarrow \Z / 2 $ by 
$$ p(\alpha_i) = \begin{cases}
0 & \textrm{ if } i \neq m\\
1 & \textrm{ if } i=m
\end{cases} $$
and $ p(\alpha_i + \cdots + \alpha_j) = p(\alpha_i) + \cdots + p(\alpha_j)$.

{\bf The superalgebra $ {\bf f}(m,n) $.} 
Let $ {\bf f}'(m,n) = {\bf f}' $ be the free, associative superalgebra with identity $ 1 $, over the field $ \Q(q) $ with even generators 
$$ \theta_1, \ldots, \theta_{m-1}, \theta_{m+1}, \ldots, \theta_{m+n-1} $$ 
and odd generator $ \theta_m$.
By abuse of notation, let $ p(\theta) $ denote the parity of a homogenous element.

By further abuse of notation, define the parity function $ p \colon I \rightarrow \Z /2 $ by $ p(i) = 0 $ for $ i \in I',I'' $ and $ p(m) =1$.
Define $ {\bf f}_{\nu}' $ to be the $ \Q(q) $ subspace spanned by monomials $ \theta_{i_1} \cdots \theta_{i_r} $ such that the number of occurrences of $ i $ in the sequence $ (i_1, \ldots, i_r) $ is equal to $ \nu_i$.
Then $ {\bf f}' = \oplus_{\nu} {\bf f}_{\nu}' $ and each $ {\bf f}_{\nu}' $ is finite-dimensional.  If $ x \in {\bf f}_{\nu}'$, then set $ |x| = \nu$.
Let $ p(\theta_{i_1} \cdots \theta_{i_r}) = p(i_1) + \cdots + p(i_r)$.

Define $ \bullet \colon \N[I] \times \N[I] \rightarrow \Z $ as follows:
\begin{equation}
\label{bullets} 
i \bullet j = \begin{cases}
-1 & \textrm{ if } i \textrm { or } j \in I' \textrm{ and } |i-j|=1\\
1 & \textrm{ if } i  \textrm { or } j \in I'' \textrm{ and } |i-j|=1\\
2 & \textrm{ if } i \in I' \textrm{ and } i=j\\
-2 & \textrm{ if } i \in I'' \textrm{ and } i=j\\
0 & \textrm{ if } |i-j| >1\\
0 & \textrm{ if } i=j=m.
\end{cases} 
\end{equation}

Extending these rules linearly gives the desired map.

Consider the vector space $ {\bf f}' \otimes {\bf f}' $.  Define the structure of a superalgebra as follows:
$$ (x_1 \otimes x_2)(x_1' \otimes x_2') = q^{-|x_2| \bullet |x_1'|} (-1)^{p(x_2)p(x_1')} (x_1 x_1' \otimes x_2 x_2'). $$

\begin{prop}
The algebra $ {\bf f}' \otimes {\bf f}' $ is associative.
\end{prop}

\begin{proof}
This follows like the proof in \cite{Lus1}.   Keeping track of the signs is easy.
\end{proof}


More generally, define the structure of an associate superalgebra on $ {\bf f}'^{\otimes b} $ by

\begin{equation}
\label{twistedbialgebra}
(x_1 \otimes \cdots \otimes x_b)(x_1' \otimes \cdots \otimes x_b') =\\
q^{-\Sigma_{i<j} |x_j| \bullet |x_i'|}(-1)^{\Sigma_{i<j} p(x_j)p(x_i')}(x_1 x_1' \otimes \cdots \otimes x_b x_b'). 
\end{equation}

\begin{prop}
The algebra $ {\bf f}'^{\otimes b}  $ is associative.
\end{prop}

\begin{proof}
The proof in \cite{Lus1} carries over almost identically.
\end{proof}

Define the algebra homomorphism $ \Delta \colon {\bf f}' \rightarrow {\bf f}' \otimes {\bf f}' $ determined by $ \Delta(\theta_i) = \theta_i \otimes 1 + 1 \otimes \theta_i $.
The next proposition may be found in \cite[Section 1.2]{Lus1}.

\begin{prop} 
\label{compositesofr}
Let $ r \colon {\bf f}' \rightarrow {\bf f}' \otimes {\bf f}' $ be an algebra homomorphism.  
\begin{enumerate}
\item $ (r \otimes 1)r \colon {\bf f}' \rightarrow {\bf f}' \otimes {\bf f}' \otimes {\bf f}' $ is an algebra homomorphism.
\item $ (1 \otimes r)r \colon {\bf f}' \rightarrow {\bf f}' \otimes {\bf f}' \otimes {\bf f}' $ is an algebra homomorphism.
\item $ (\Delta \otimes 1)\Delta = (1 \otimes \Delta)\Delta$.
\end{enumerate}
\end{prop}

\begin{proof}
The first two parts are relatively straightforward calculations.  For the third part, note that both maps are equal on generators $ \theta_i$.  The first two parts imply that these maps are algebra homomorphisms so they are equal on all elements in $ f'$.
\end{proof}

Given $ \Delta $ as above, define $ {\Delta}^b $ inductively with $ {\Delta}^1 = \Delta$ and $ {\Delta}^b = (1 \otimes \Delta){\Delta}^{b-1}$.  Proposition ~\ref{compositesofr} gives that 
$ {\Delta}^b = (\Delta \otimes 1){\Delta}^{b-1}$.
As usual, define the divided power $ \theta_i^{(p)} = \frac{\theta_i^{p}}{[p]!}$ where $ [p]=\frac{q^{p}-q^{-p}}{q-q^{-1}} $ and $ [p]!=[p][p-1]\cdots[1] $.

We now collect some calculations of $ \Delta^b $ applied to various elements of $ {\bf f}'$.

\begin{prop} 
\label{Deltathetaip}
If $ i \neq m$, then
$$ \Delta(\theta_i^{(p)}) = \ \sum_{t+t'=p} q^{-\frac{i \bullet i}{2} t t'} \theta_i^{(t)} \otimes \theta_i^{(t')}. $$
\end{prop}

\begin{proof}
This follows exactly as in \cite[Lemma 1.4.2]{Lus1}.  It is an induction on $ p$.
\end{proof}

\begin{prop}
\label{Deltathetamp}
\begin{enumerate}
\item $ \Delta(\theta_m^{2p}) = \ \sum_{\gamma=0}^p {p \brack \gamma} \theta_m^{2p-2\gamma} \otimes \theta_m^{2\gamma}. $
\item $ \Delta(\theta_m^{2p+1}) = \ \sum_{\gamma=0}^p {p \brack \gamma} \theta_m^{2p+1-2\gamma} \otimes \theta_m^{2\gamma} + 
\ \sum_{\gamma=0}^p {p \brack \gamma} \theta_m^{2p-2\gamma} \otimes \theta_m^{2\gamma+1}. $
\end{enumerate}
\end{prop}

\begin{proof}
The first part is proved by induction on $ p$.  The second part easily follows from the first.
\end{proof}

The next proposition generalizes Proposition ~\ref{Deltathetaip}.

\begin{prop}
If $ i \neq m$, then
$$ \Delta^b(\theta_i^{(p)}) = \sum_{\gamma_1 + \cdots + \gamma_{b+1} = p} q^{\frac{-i \bullet i}{2} \sum_{j \neq k} \gamma_j \gamma_k} \theta_i^{(\gamma_1)} \otimes \cdots \otimes \theta_i^{(\gamma_{b+1})}. $$
\end{prop}

\begin{proof}
This is an easy induction on $ b$.
\end{proof}

The next two propositions generalize Proposition ~\ref{Deltathetamp}.

\begin{prop}
\label{Deltaeven}
$$ \Delta^b(\theta_m^{(2p)}) = \sum_{\gamma_1 + \cdots + \gamma_{b+1} = p} \frac{{p \brack \gamma_1, \ldots, \gamma_{b+1}}}{{2p \brack 2\gamma_1, \ldots, 2\gamma_{b+1}}}
\theta_m^{(2\gamma_1)} \otimes \cdots \otimes \theta_m^{(2 \gamma_{b+1})}.  $$
\end{prop}

\begin{proof}
This is an induction on $ b$.
\end{proof}

\begin{prop}
$$ \Delta^b(\theta_m^{(2p+1)}) = \sum_{\gamma_1 + \cdots + \gamma_{b+1} = p} \sum_{k=1}^{b+1} \frac{{p \brack \gamma_1, \ldots, \gamma_{b+1}}}{{2p+1 \brack 2\gamma_1, \ldots, 2\gamma_{k-1}, 2\gamma_k+1, \ldots, 2\gamma_{b+1}}}
\theta_m^{(2\gamma_1)} \otimes \cdots \otimes \theta_m^{(2 \gamma_{k-1})} \otimes \theta_m^{(2\gamma_k+1)} \otimes \cdots \otimes \theta_m^{(2 \gamma_{b+1})}.  $$
\end{prop}

\begin{proof}
This follows from Proposition ~\ref{Deltaeven}.
\end{proof}

\subsection{The bilinear form $ ( , ) $ and its radical $ \mathcal{I} $.}

\begin{prop}
\label{formproperties}
There exists a unique bilinear map $ ( , ) \colon {\bf f}' \times {\bf f}' \rightarrow \Q(q) $ such that
\begin{itemize}
\item $ (1,1) = 1 $
\item $ (x, y'y'') = (\Delta(x), y' \otimes y'') $
\item $ (x'x'', y) = (x' \otimes x'', \Delta(y)) $
\item $ (\theta_i, \theta_j) = \frac{\delta_{i,j}}{1-q^{i \bullet i}} $ if $ i \neq m $
\item $ (\theta_m, \theta_m) =1$
\end{itemize}
and
$ (\theta_{i_1} \otimes \cdots \otimes \theta_{i_r}, \theta_{j_1} \otimes \cdots \otimes \theta_{j_r}) = (\theta_{i_1}, \theta_{j_1}) \cdots (\theta_{i_r},\theta_{j_r}) $.
\end{prop}

\begin{proof}
The proof follows as in \cite[Proposition 1.2.3]{Lus1}.
\end{proof}

Let $ \mathcal{I} $ denote the radical of the form $ ( , ) $.  

\begin{prop}
\label{hopfideal}
The radical $ \mathcal{I} $ is a two-sided ideal of $ {\bf f}'$.  Furthermore, $ \Delta(\mathcal{I}) \subset \mathcal{I} \otimes {\bf f}' + {\bf f}' \otimes \mathcal{I} $.  
\end{prop}

\begin{proof}
This follows as in \cite[Sections 1.2.4-1.2.6]{Lus1}.
\end{proof}

Define $ {\bf f}(m,n) $ to be $ {\bf f}'(m,n) / \mathcal{I}$.
We will often denote this algebra by $ {\bf f}$.
Proposition ~\ref{hopfideal} shows that $ \Delta $ descends to a homomorphism $ \Delta \colon {\bf f} \rightarrow {\bf f} \otimes {\bf f} $.

\begin{prop}
\label{formdividedpowers}
$$ (\theta_i^{(p)}, \theta_i^{(p)}) =\begin{cases} \Pi_{s=1}^p \frac{1}{1-q^{2s}} & \textrm{ if } \ i \in I' , \\
                         \Pi_{s=1}^p \frac{1}{1-q^{-2s}} & \textrm{ if } i \in I'',\\
                         0 & \textrm{ if } i=m \textrm{ and } p>1,\\ 
                         1 & \textrm{ if } i=m \textrm{ and } p=1.
                         \end{cases} $$ 
\end{prop}

\begin{proof}
If $ i \neq m$, then this is the same proof as in \cite[Lemma 1.4.4]{Lus1}.
The pairing $ (\theta_m^p, \theta_m^p) $ is equal to zero for $ p>1 $ by induction on $ p$.
By Proposition ~\ref{formproperties}, 
$$ (\theta_m^2, \theta_m^2) = (\theta_m \otimes \theta_m, \theta_m^2 \otimes 1 + 1 \otimes \theta_m^2)=0 $$
which gives the base case.
The induction step, follows easily from:
$$ (\theta_m^{p'+p''}, \theta_m^{p'+p''}) = (\theta_m^{p'} \otimes \theta_m^{p''}, \Delta(\theta_m^{p'}) \Delta(\theta_m^{p''})). $$
\end{proof}

\begin{prop}
\label{gen1}
$ \theta_m^2 \in \mathcal{I}$.
\end{prop}

\begin{proof}
It is only necessary to check that $ \theta_m^2 $ paired with itself is zero:
$$ (\theta_m^2, \theta_m^2) = (\theta_m \otimes \theta_m, \Delta(\theta_m) \Delta(\theta_m)) = (\theta_m \otimes \theta_m, \theta_m^2 \otimes 1 + 1 \otimes \theta_m^2) = 0.$$
\end{proof}

\begin{prop}
\label{gen2}
If $ |i-j| >1$, then $ \theta_i \theta_j - \theta_j \theta_i \in \mathcal{I}$.
\end{prop}

\begin{proof}
\begin{align*}
(\theta_i \theta_j, \theta_i \theta_j-\theta_j \theta_i) &= (\theta_i \otimes \theta_j - \theta_j \otimes \theta_i, (\theta_i \otimes 1 + 1 \otimes \theta_i)(\theta_j \otimes 1 + 1 \otimes \theta_j))\\
&= (\theta_i \otimes \theta_j - \theta_j \otimes \theta_i, \theta_i \otimes \theta_j + (-1)^{p(i)p(j)} \theta_j \otimes \theta_i).
\end{align*}
Note that we have eliminated terms which obviously make the inner product vanish.
Since $ p(i)p(j) =0$, this is equal to
$$ (\theta_i, \theta_i) (\theta_j, \theta_j) - (\theta_i, \theta_i) (\theta_j, \theta_j) =0. $$
Similarly,
$ (\theta_j \theta_i, \theta_i \theta_j-\theta_j \theta_i) =0, $
so $ \theta_i \theta_j - \theta_j \theta_i \in \mathcal{I}$.
\end{proof}

\begin{prop}
\label{gen3calculations}
Let $ \nu = 1(m-1)+2(m)+1(m+1)$.  The pairings on $ {\bf f}_{\nu}' $ are:
\begin{enumerate}
\item $ (\theta_m \theta_{m-1} \theta_{m+1} \theta_m, \theta_m \theta_{m-1} \theta_{m+1} \theta_m)=0$
\item $ (\theta_m \theta_{m-1} \theta_{m+1} \theta_m, \theta_m \theta_{m-1} \theta_{m} \theta_{m+1})=\frac{q^{-1}}{1-q^{-2}}$
\item $ (\theta_m \theta_{m-1} \theta_{m+1} \theta_m, \theta_{m-1} \theta_{m} \theta_{m+1} \theta_{m})=\frac{-q^{-1}}{1-q^{-2}}$
\item $ (\theta_m \theta_{m-1} \theta_{m+1} \theta_m, \theta_{m+1} \theta_{m} \theta_{m-1} \theta_{m})=\frac{q^{-1}}{1-q^{-2}}$
\item $ (\theta_m \theta_{m-1} \theta_{m+1} \theta_m, \theta_{m} \theta_{m+1} \theta_{m} \theta_{m-1})=\frac{-q^{-1}}{1-q^{-2}}$
\item $ (\theta_m \theta_{m-1} \theta_{m} \theta_{m+1}, \theta_{m} \theta_{m-1} \theta_{m} \theta_{m+1})=\frac{1}{1-q^{-2}}$
\item $ (\theta_m \theta_{m-1} \theta_{m} \theta_{m+1}, \theta_{m-1} \theta_{m} \theta_{m+1} \theta_{m})=0$
\item $ (\theta_m \theta_{m-1} \theta_{m} \theta_{m+1}, \theta_{m+1} \theta_{m} \theta_{m-1} \theta_{m})=\frac{q^{-2}}{1-q^{-2}}$
\item $ (\theta_m \theta_{m-1} \theta_{m} \theta_{m+1}, \theta_{m} \theta_{m+1} \theta_{m} \theta_{m-1})=0$
\item $ (\theta_{m-1} \theta_{m} \theta_{m+1} \theta_{m}, \theta_{m-1} \theta_{m} \theta_{m+1} \theta_{m})=\frac{1}{1-q^2}$
\item $ (\theta_{m-1} \theta_{m} \theta_{m+1} \theta_{m}, \theta_{m+1} \theta_{m} \theta_{m-1} \theta_{m})=0$
\item $ (\theta_{m-1} \theta_{m} \theta_{m+1} \theta_{m}, \theta_{m} \theta_{m+1} \theta_{m} \theta_{m-1})=\frac{q^2}{1-q^2}$
\item $ (\theta_{m+1} \theta_{m} \theta_{m-1} \theta_{m}, \theta_{m+1} \theta_{m} \theta_{m-1} \theta_{m})=\frac{1}{1-q^{-2}}$
\item $ (\theta_{m+1} \theta_{m} \theta_{m-1} \theta_{m}, \theta_{m} \theta_{m+1} \theta_{m} \theta_{m-1})=0$
\item $ (\theta_{m} \theta_{m+1} \theta_{m} \theta_{m-1}, \theta_{m} \theta_{m+1} \theta_{m} \theta_{m-1})=\frac{1}{1-q^2}$.
\end{enumerate}
\end{prop}

\begin{proof}
The calculations for all the parts are similar so we give only the proof of the first equation.
\begin{align*}
(\theta_m \theta_{m-1} \theta_{m+1} \theta_m, \theta_m \theta_{m-1} \theta_{m+1} \theta_m) &= (\theta_m \otimes \theta_{m-1} \theta_{m+1} \theta_m, \Delta(\theta_m) \Delta(\theta_{m-1})\Delta(\theta_{m+1})\Delta(\theta_m)) \\
&= (\theta_m \otimes \theta_{m-1} \theta_{m+1} \theta_m, \theta_m \otimes \theta_{m-1} \theta_{m+1} \theta_m - \theta_m \otimes \theta_{m} \theta_{m-1} \theta_{m+1})\\
&=(\theta_{m-1} \otimes \theta_{m+1} \theta_m, \theta_{m-1} \otimes \theta_{m+1} \theta_m - q \theta_{m-1} \otimes \theta_m \theta_{m+1})\\
&=\frac{1}{1-q^{2}}(\theta_{m+1} \otimes \theta_m, \theta_{m+1} \otimes \theta_m - (q)(q^{-1}) \theta_{m+1} \otimes \theta_m)\\
&=0.
\end{align*}
\end{proof}

\begin{prop}
\label{gen3}
$$ (q+q^{-1})\theta_{m} \theta_{m-1} \theta_{m+1} \theta_m  
- \theta_{m-1} \theta_{m} \theta_{m+1} \theta_{m} 
- \theta_{m} \theta_{m-1} \theta_{m} \theta_{m+1}
- \theta_{m} \theta_{m+1} \theta_{m} \theta_{m-1}
- \theta_{m+1} \theta_{m} \theta_{m-1} \theta_{m}
\in \mathcal{I}. $$
\end{prop}

\begin{proof}
It suffices to show that the above element paired with $ \theta_{m} \theta_{m-1} \theta_{m+1} \theta_m$, 
$ \theta_{m-1} \theta_{m} \theta_{m+1} \theta_{m} $, 
\newline
$ \theta_{m} \theta_{m-1} \theta_{m} \theta_{m+1} $, 
$ \theta_{m} \theta_{m+1} \theta_{m} \theta_{m-1} $, and
$ \theta_{m+1} \theta_{m} \theta_{m-1} \theta_{m} $ are all zero.
This now follows easily from Proposition ~\ref{gen3calculations}.
\end{proof}

The proof of the next proposition uses Proposition ~\ref{formproperties} and is similar to the proof of Proposition ~\ref{gen3calculations}.

\begin{prop}
\label{gen4calculations}
Let $ i, i+1 \in I $, $ i \neq m$, and $ \nu = 2i+1(i+1) $.
The pairings on $ {\bf f}'_{\nu} $ are:
\begin{enumerate}
\item $$ (\theta_i^2 \theta_{i+1}, \theta_i^2 \theta_{i+1}) =
\begin{cases}  \frac{1+q^{-i \bullet i}}{(1-q^{i \bullet i})^3} & \textrm{ if } p(i+1)=0 , \\
                          \frac{1+q^{-i \bullet i}}{(1-q^{i \bullet i})^2} & \textrm{ if } p(i+1)=1
\end{cases} $$ 
\item $$ (\theta_i^2 \theta_{i+1}, \theta_i \theta_{i+1} \theta_i) =
\begin{cases}  \frac{q+q^{-1}}{(1-q^{i \bullet i})^3} & \textrm{ if } p(i+1)=0 , \\
                          \frac{q+q^{-1}}{(1-q^{i \bullet i})^2} & \textrm{ if } p(i+1)=1
\end{cases} $$ 
\item $$ (\theta_i^2 \theta_{i+1}, \theta_{i+1} \theta_i^2) =
\begin{cases}  \frac{q^{-i \bullet (i+1)}(q+q^{-1})}{(1-q^{i \bullet i})^3} & \textrm{ if } p(i+1)=0 , \\
                          \frac{q^{-i \bullet (i+1)}(q+q^{-1})}{(1-q^{i \bullet i})^2} & \textrm{ if } p(i+1)=1
\end{cases} $$ 
\item $$ (\theta_i \theta_{i+1} \theta_i, \theta_{i} \theta_{i+1} \theta_{i}) =
\begin{cases}  \frac{2}{(1-q^{i \bullet i})^3} & \textrm{ if } p(i+1)=0 , \\
                          \frac{2}{(1-q^{i \bullet i})^2} & \textrm{ if } p(i+1)=1
\end{cases} $$ 
\item $$ (\theta_i \theta_{i+1} \theta_i, \theta_{i+1} \theta_{i}^2) =
\begin{cases}  \frac{q+q^{-1}}{(1-q^{i \bullet i})^3} & \textrm{ if } p(i+1)=0 , \\
                          \frac{q+q^{-1}}{(1-q^{i \bullet i})^2} & \textrm{ if } p(i+1)=1
\end{cases} $$ 
\item $$ (\theta_{i+1} \theta_i^2, \theta_{i+1} \theta_{i}^2) =
\begin{cases}  \frac{1+q^{-i \bullet i}}{(1-q^{i \bullet i})^3} & \textrm{ if } p(i+1)=0 , \\
                          \frac{1+q^{-i \bullet i}}{(1-q^{i \bullet i})^2} & \textrm{ if } p(i+1)=1.
\end{cases} $$ 
\end{enumerate}
\end{prop}

\begin{prop}
\label{gen4}
Assume $ i \neq m$ and $ i \pm 1 \in I$.  Then
$$ (q+q^{-1})\theta_i \theta_{i \pm 1} \theta_i - \theta_i^2 \theta_{i \pm 1} - \theta_{i \pm 1} \theta_i^2 \in \mathcal{I}. $$
\end{prop} 

\begin{proof}
It suffices to show that 
$$ (q+q^{-1})\theta_i \theta_{i \pm 1} \theta_i - \theta_i^2 \theta_{i \pm 1} - \theta_{i \pm 1} \theta_i^2 $$ paired with
$ \theta_i^2 \theta_{i \pm1} $, $ \theta_i \theta_{i \pm 1} \theta_i$, and $ \theta_{i \pm 1} \theta_i^2 $ are all zero.  The plus case follows easily from Proposition ~\ref{gen4calculations}.  The minus case follows from analogous 
calculations.
\end{proof}

\subsection{PBW basis}
We now look to construct a PBW basis for the algebra $ {\bf f}$.  

Order the set of positive roots as follows:
\begin{eqnarray*}
&\alpha_1 < \alpha_1 + \alpha_2 < \cdots < \alpha_1 + \cdots + \alpha_{m+n-1} < \\
&\alpha_2 < \alpha_2 + \alpha_3 < \cdots < \alpha_2 + \cdots + \alpha_{m+n-1} < \\
&\cdots \\
&< \alpha_{m+n-1}
\end{eqnarray*}

Now to each $ \alpha \in \Delta^+$, we define an element $ \theta_{\alpha} \in f' $.
Set $ \theta_{\alpha_i} = \theta_i$.  
Let $ \alpha = \alpha_i + \cdots + \alpha_j $ and
$ \alpha' = \alpha_i + \cdots + \alpha_{j-1}$.
Then set 
\begin{equation}
\label{rootvector}
\theta_{\alpha} = \theta_{\alpha'} \theta_j - (-1)^{p({\alpha'})p(\alpha_j)} q^{\alpha' \bullet (\alpha_{j})} \theta_j \theta_{\alpha'}.
\end{equation}

Note that the definition of $ \theta_{\alpha} $ above differs from the definition of the root vector given in ~\cite{Yam} by a sign in the exponent of $q$.

\begin{lemma}
\label{commlemma}
Let $ \alpha = \alpha_{j-1}+\alpha_j+\alpha_{j+1}$. Then
$ \theta_{\alpha} \theta_j - (-1)^{p(\alpha)p(j)} \theta_j \theta_{\alpha} \in \mathcal{I}$.
\end{lemma}

\begin{proof}
Write $ \theta_{\alpha} $ in terms of $ \theta_{j-1}, \theta_j$ and $ \theta_{j+1} $ using Equation ~\ref{rootvector}.
The lemma follows from a routine calculation and Proposition ~\ref{gen3} if $ j=m $ or Proposition ~\ref{gen4} if $ j \neq m$.
\end{proof}

\begin{prop}
\label{pbw1}
Let $ \alpha = \alpha_i + \cdots + \alpha_j, \beta = \alpha_{k} + \cdots + \alpha_l$.  Assume $ k > j+1$.  Then
$ \theta_{\alpha} \theta_{\beta} - \theta_{\beta} \theta_{\alpha} \in \mathcal{I}$.
\end{prop}

\begin{proof}
If $ \theta_r $ is a term appearing in $ \theta_{\alpha} $ and $ \theta_s $ is a term appearing in $ \theta_{\beta}$, then 
$ \theta_r \theta_s - \theta_s \theta_r \in \mathcal{I}$.  
One proceeds by induction on $ l $ with $ k=l $ being the base case. 
Let $ \beta' = \alpha_k + \cdots + \alpha_{l-1} $.
By Equation ~\ref{rootvector}, replace $ \theta_{\beta} $ in 
$  \theta_{\alpha} \theta_{\beta} - \theta_{\beta} \theta_{\alpha} $
by 
$$ \theta_{\beta'} \theta_{\alpha_l} - (-1)^{p({\beta'})p(\alpha_l)} q^{\beta' \bullet (\alpha_{l})} \theta_{\alpha_l} \theta_{\beta'}. $$
Now invoke the induction hypothesis to finish the proof.
\end{proof}

\begin{prop}
\label{pbw2}
Let $ \alpha = \alpha_i + \cdots + \alpha_j $ and $ \beta = \alpha_{j+1} + \cdots + \alpha_l$.  Then
$$ \Gamma \ := \theta_{\alpha} \theta_{\beta} - (-1)^{p(\alpha)p(\beta)} q^{\alpha \bullet \beta} \theta_{\beta} \theta_{\alpha} - \theta_{\alpha+\beta} \in \mathcal{I}. $$
\end{prop}

\begin{proof}
We proceed by induction on $ l $ with $ k=l $ as the base case which is true by definition.
Let $ \beta' = \alpha_k + \cdots + \alpha_{l-1} $.  Then by induction,
$$ \theta_{\alpha} \theta_{\beta'} - (-1)^{p(\alpha)p(\beta')} q^{\alpha \bullet \beta'} \theta_{\beta'} \theta_{\alpha} - \theta_{\alpha+\beta'} \in \mathcal{I}. $$
By definiton,
\begin{equation}
\label{pbw2eq1}
\theta_{\beta} = \theta_{\beta'} \theta_l - (-1)^{p(\beta')p(\alpha_l)}q^{\beta' \bullet \alpha_l} \theta_l \theta_{\beta'}\\
\end{equation}
\begin{equation}
\label{pbw2eq2}
\theta_{\alpha+ \beta} = \theta_{\alpha+\beta'}\theta_l - (-1)^{p(\alpha+\beta')p(\alpha_l)}q^{\alpha+ \beta' \bullet \alpha_l} \theta_l \theta_{\alpha+ \beta'}.
\end{equation}
Using Equations ~\ref{pbw2eq1} and ~\ref{pbw2eq2} and rearranging terms gives:
\begin{equation}
\label{pbw2eq3}
\begin{split}
\Gamma = &-(-1)^{p(\beta')p(\alpha_l)} q^{\beta' \bullet \alpha_l} \theta_{\alpha} \theta_l \theta_{\beta'}
+(-1)^{p(\alpha)p(\beta)+p(\alpha_l)p(\beta')}q^{\alpha \bullet \beta + \alpha_l \bullet \beta'} \theta_l \theta_{\beta'} \theta_{\alpha}\\
&+(-1)^{p(\alpha+\beta')p(\alpha_l)} q^{(\alpha+\beta') \bullet \alpha_l} \theta_l \theta_{\alpha+\beta'}
+\theta_{\alpha} \theta_{\beta'} \theta_l - (-1)^{p(\alpha)p(\beta)} q^{\alpha \bullet \beta} \theta_{\beta'} \theta_l \theta_{\alpha} - \theta_{\alpha+\beta'}\theta_l.
\end{split}
\end{equation}
Note that $ \alpha \bullet \beta = \alpha \bullet \beta' $ and $ p(\alpha)p(\beta)=p(\alpha)p(\beta') \text{ mod } 2 $.  Then using Proposition ~\ref{pbw1} and the induction hypothesis, 
the sum of the last three terms in Equation ~\ref{pbw2eq3} is in $ \mathcal{I}$.

The first three terms in Equation ~\ref{pbw2eq3} is equal to
$$ -(-1)^{p(\beta')p(\alpha_l)}q^{\beta \bullet \alpha_l}[\theta_l \theta_{\alpha} \theta_{\beta'} -(-1)^{p(\alpha)p(\beta)}q^{\alpha \bullet \beta} \theta_l \theta_{\beta'} \theta_{\alpha}-(-1)^{p(\alpha)p(\alpha_l)}
q^{\alpha \bullet \alpha_l} \theta_l \theta_{\alpha+\beta'}]+\gamma $$
where $ \gamma \in \mathcal{I}$.  Note that $ \gamma $ arises when applying Proposition ~\ref{pbw1} to the first term in Equation ~\ref{pbw2eq3}.
Since $ p(\alpha)p(\beta) = p(\alpha)p(\beta') $, $ \alpha \bullet \beta = \alpha \bullet \beta'$, $ \alpha \bullet \alpha_l =0$, and
$ p(\alpha)p(\alpha_l)=0 \text{ mod } 2 $, the sum of the first three terms in Equation ~\ref{pbw2eq3} is now also in $ \mathcal{I} $ by the induction hypothesis.
\end{proof}

\begin{prop}
\label{pbw3}
Let $ i < k \leq l < j$, $ \alpha = \alpha_i + \cdots + \alpha_j$, and $ \beta = \alpha_k + \cdots + \alpha_l$.  Then
$$ \theta_{\alpha} \theta_{\beta} - (-1)^{p(\alpha)p(\beta)} \theta_{\beta} \theta_{\alpha} \in \mathcal{I}.$$
\end{prop}

\begin{proof}
This is an induction on $ l $ with $ k=l $ as the base case.

\emph{Step 1.} Assume $ k=l$.
Let $ \alpha' = \alpha_i + \cdots + \alpha_{j-1}$.  We must show that 
$$ \gamma \ := \theta_{\alpha} \theta_k - (-1)^{p(\alpha)p(\beta)} \theta_k \theta_{\alpha} \in \mathcal{I}. $$
We do this by induction on $ j$.
The base case is when $ j=k+1$.
Then $ \gamma = \theta_{\alpha} \theta_{j-1} - (-1)^{p(\alpha)p(\beta)} \theta_{j-1} \theta_{\alpha}$. 
Let $ \alpha'' = \alpha_i + \cdots + \alpha_{j-3} $ and $ \alpha''' = \alpha_{j-2}+\alpha_{j-1}+\alpha_j$.
Then $\gamma$ is in $ \mathcal{I} $ if and only if 
\begin{equation}
\label{pbw3eq1}
\theta_{\alpha''} \theta_{\alpha'''} \theta_{{j-1}} - (-1)^{p(\alpha)p(\beta)} \theta_{j-1} \theta_{\alpha} - (-1)^{p(\alpha_{j-3})p(\alpha_{j-2})} q^{\alpha_{j-3} \bullet \alpha_{j-2}} \theta_{\alpha'''} \theta_{\alpha''} \theta_{j-1} \in \mathcal{I}.
\end{equation}
By Proposition ~\ref{pbw1} and Lemma ~\ref{commlemma}, the element in Equation ~\ref{pbw3eq1} is in $ \mathcal{I} $ if and only if
\begin{equation}
\label{pbw3eq2}
\theta_{{j-1}} \theta_{\alpha''} \theta_{\alpha'''} - (-1)^{p(\alpha)p(\beta)} \theta_{j-1} \theta_{\alpha} - (-1)^{p(\alpha_{j-3})p(\alpha_{j-2})} q^{\alpha_{j-3} \bullet \alpha_{j-2}} \theta_{j-1} \theta_{\alpha'''} \theta_{\alpha''}  \in \mathcal{I}.
\end{equation}
This now follows from Proposition ~\ref{pbw2}.

So now suppose $ j > k+1$.  
Then 
$$ \gamma = \theta_{\alpha'} \theta_j \theta_k - (-1)^{p(\alpha_{j-1})p(\alpha_j)} q^{\alpha_{j-1} \bullet \alpha_j} \theta_j \theta_{\alpha'} \theta_k - (-1)^{p(\alpha)p(\beta)} \theta_k \theta_{\alpha}.$$
By induction, this is in $ \mathcal{I} $ if and only if 
$$ \theta_k \theta_{\alpha'} \theta_j - (-1)^{p(\alpha_{j-1})p(\alpha_j)} q^{\alpha_{j-1} \bullet \alpha_j} \theta_k \theta_j \theta_{\alpha'} - (-1)^{p(\alpha)p(\beta)} \theta_k \theta_{\alpha} \in \mathcal{I}.$$
Since $ \theta_{\alpha'} \theta_j - (-1)^{p(\alpha_{j-1})p(\alpha_j)} q^{\alpha_{j-1} \bullet \alpha_j} \theta_j \theta_{\alpha'} - \theta_{\alpha} =0$, 
we are done with this induction.

\emph{Step 2.} 
Now assume $ l-k \geq 1$.  
Let $ \beta' = \alpha_k + \cdots + \alpha_{l-1} $.
Then $ \theta_{\alpha} \theta_{\beta} -(-1)^{p(\alpha)p(\beta)} \theta_{\beta} \theta_{\alpha} \in \mathcal{I} $ if and only if
$$ \theta_{\alpha} \theta_{\beta'} \theta_l - (-1)^{p(\alpha_{l-1})p(\alpha_l)}q^{\alpha_{l-1} \bullet \alpha_l} \theta_{\alpha} \theta_l \theta_{\beta'}-(-1)^{p(\alpha)p(\beta)} \theta_{\beta}\theta_{\alpha} \in \mathcal{I}. $$
By induction, the above element is in $ \mathcal{I} $ if and only if
$$  \theta_{\beta'} \theta_l \theta_{\alpha} - (-1)^{p(\alpha_{l-1})p(\alpha_l)}q^{\alpha_{l-1} \bullet \alpha_l} \theta_l \theta_{\beta'} \theta_{\alpha} - (-1)^{p(\alpha)p(\beta)} \theta_{\beta} \theta_{\alpha} \in \mathcal{I}. $$
By definition, the above element is zero and hence in $ \mathcal{I}$.
\end{proof}

\begin{prop}
\label{pbw4}
Let $ j>l$, $ \alpha = \alpha_i + \cdots + \alpha_j$, and $ \beta = \alpha_i+ \cdots + \alpha_l$.  
\begin{enumerate}
\item If $ i \neq m$, then
$ A \ := \theta_{\alpha} \theta_{\beta} - (-1)^{p(i)p(i+1)} q^{i \bullet (i+1)} \theta_{\beta} \theta_{\alpha} \in \mathcal{I}.$
\item If $ i =m $, then
$ B\ := \theta_{\alpha} \theta_{\beta} + (-1)^{p(i)p(i+1)} q^{-i \bullet (i+1)} \theta_{\beta} \theta_{\alpha} \in \mathcal{I}.$
\end{enumerate}
\end{prop}

\begin{proof}
First consider the case that $ l=i$.  We prove the proposition in this case by induction on $ j $ with $ j=i+1$ as the base case.
Using Equation ~\ref{rootvector}, write $ \theta_{\alpha_i+\alpha_{i+1}} $ in terms of $ \theta_i $ and $ \theta_{i+1} $ and substitute these quantities into the proposition.  If $ i\neq m$, then Proposition ~\ref{gen4}
shows that $ A \in \mathcal{I}$.  If $ i=m$, then Proposition ~\ref{gen1} shows that $ B \in \mathcal{I}$.  This gives the base case of the induction for the case $ l=i$.

Now consider the more general case $ \alpha=\alpha_i+\cdots+\alpha_j $ but again $ \beta=\alpha_i$. 
Write $ \theta_{\alpha} = \theta_{\alpha'} \theta_{\alpha_j} - (-1)^{p(j-1)p(j)} q^{\alpha_{j-1} \bullet \alpha_j} \theta_{\alpha_j} \theta_{\alpha'} $,
where $ \alpha' = \alpha_i + \cdots + \alpha_{j-1}$.  Using Proposition ~\ref{pbw1} and the induction hypothesis, we easily verify that $ A, B \in \mathcal{I} $, (again in the case $ l=i$).

Finally, consider the case that $ l \geq i$.  We now proceed by induction on $ l $ with $ l=i$ as the base case which was verified above.
In order to prove the induction step, let $ \beta' = \alpha_i+ \cdots + \alpha_{l-1} $ and using Equation ~\ref{rootvector} write
$ \theta_{\beta} = \theta_{\beta'} \theta_l - (-1)^{p(l-1)p(l)} q^{\alpha_{l-1} \bullet \alpha_l} \theta_l \theta_{\beta'}$.
Substitute this into the expressions for $ A $ and $ B $ and use the induction hypothesis to complete the proof.
\end{proof}

\begin{prop}
\label{pbw5}
Let $ i<k$, $ \alpha = \alpha_i + \cdots + \alpha_j$, and $ \beta = \alpha_k+ \cdots + \alpha_j$.  
\begin{enumerate}
\item If $ j \neq m$, then
$ A:= \theta_{\alpha} \theta_{\beta} - (-1)^{p(j-1)p(j)} q^{-(j-1) \bullet j} \theta_{\beta} \theta_{\alpha} \in \mathcal{I}.$
\item If $ j =m $, then
$ B:= \theta_{\alpha} \theta_{\beta} + (-1)^{p(j-1)p(j)} q^{(j-1) \bullet j} \theta_{\beta} \theta_{\alpha} \in \mathcal{I}.$
\end{enumerate}
\end{prop}

\begin{proof}
This is very similar to the proof of Proposition ~\ref{pbw4}.
\end{proof}
 
\begin{prop}
\label{pbw6}
Let $ i < k < j < l $, $ \alpha = \alpha_i + \cdots + \alpha_j $ and $ \beta = \alpha_k + \cdots + \alpha_l $.  Furthermore, let
$ \gamma = \alpha_i + \cdots + \alpha_l $ and $ \gamma' = \alpha_k + \cdots + \alpha_j$.
\begin{enumerate}
\item If $ j \neq m$, then
$ \theta_{\alpha} \theta_{\beta} - \theta_{\beta} \theta_{\alpha} - (q^{-j \bullet (j+1)}-q^{j \bullet (j+1)}) \theta_{\gamma} \theta_{\gamma'} \in \mathcal{I}$.
\item If $ j =m$, then
$ \theta_{\alpha} \theta_{\beta} + \theta_{\beta} \theta_{\alpha} + (q+q^{-1}) \theta_{\gamma} \theta_{\gamma'}  \in \mathcal{I}$.
\end{enumerate}
\end{prop}

\begin{proof}
Using Propositions ~\ref{pbw2}, ~\ref{pbw3}, ~\ref{pbw4}, and ~\ref{pbw5}, the following equalities are obtained modulo $ \mathcal{I}$:
\begin{align*}
\theta_{\beta} \theta_{\alpha} &= \theta_{\alpha_k + \cdots + \alpha_j} \theta_{\alpha_{j+1} + \cdots + \alpha_l} \theta_{\alpha} - q^{j \bullet (j+1)} 
\theta_{\alpha_{j+1}+\cdots+\alpha_l} \theta_{\alpha_k + \cdots + \alpha_j} \theta_{\alpha}\\
&= q^{-j \bullet (j+1)} \theta_{\alpha_k + \cdots + \alpha_j} \theta_{\alpha} \theta_{\alpha_{j+1} + \cdots + \alpha_l}
-q^{-j \bullet (j+1)} \theta_{\alpha_k + \cdots + \alpha_j} \theta_{\gamma} - q^{j \bullet (j+1)} \theta_{\alpha_{j+1} + \cdots + \alpha_l} \theta_{\alpha_k + \cdots + \alpha_j} \theta_{\alpha}.
\end{align*}
Now assume that $ j \neq m $.  Then $ (j-1) \bullet j = j \bullet (j+1)$.   Then continuing from above,
\begin{align*}
\theta_{\beta} \theta_{\alpha} &= 
\theta_{\alpha} \theta_{\alpha_k + \cdots + \alpha_j} \theta_{\alpha_{j+1} + \cdots + \alpha_l} - q^{-j \bullet (j+1)} \theta_{\gamma} \theta_{\gamma'}
-q^{j \bullet (j+1)} \theta_{\alpha} \theta_{\alpha_{j+1} + \cdots + \alpha_l} \theta_{\alpha_k + \cdots + \alpha_j} + q^{j \bullet (j+1)} \theta_{\gamma} \theta_{\gamma'}\\
&= \theta_{\alpha} \theta_{\beta} + (q^{j \bullet (j+1)} - q^{-j \bullet (j+1)}) \theta_{\gamma} \theta_{\gamma'}
\end{align*}
Now assume $ j=m$.  Then $ (j-1) \bullet j = -1 $ and $ j \bullet (j+1) = 1$.  Then continuing from before,
\begin{align*}
\theta_{\beta} \theta_{\alpha} &=
-\theta_{\alpha} \theta_{\alpha_k+\cdots+\alpha_j} \theta_{\alpha_{j+1}+\cdots+\alpha_l}
-q^{-1} \theta_{\gamma} \theta_{\gamma'} + q^{2} \theta_{\alpha_{j+1}+\cdots+\alpha_l} \theta_{\alpha} \theta_{\alpha_k+\cdots+\alpha_j}\\
&=
-\theta_{\alpha} \theta_{\alpha_k+\cdots+\alpha_j} \theta_{\alpha_{j+1}+\cdots+\alpha_l}
-q^{-1} \theta_{\gamma} \theta_{\gamma'} + q^{} \theta_{\alpha} \theta_{\alpha_{j+1}+\cdots+\alpha_l}  \theta_{\alpha_k+\cdots+\alpha_j} -q \theta_{\gamma} \theta_{\gamma'} \\
&=
-\theta_{\alpha} \theta_{\alpha_k+\cdots+\alpha_j} \theta_{\alpha_{j+1}+\cdots+\alpha_l}
-q^{-1} \theta_{\gamma} \theta_{\gamma'} + \theta_{\alpha} \theta_{\alpha_k+\cdots+\alpha_j} \theta_{\alpha_{j+1}+\cdots+\alpha_l} - \theta_{\alpha} \theta_{\beta}-q \theta_{\gamma} \theta_{\gamma'} \\
&= -(q+q^{-1}) \theta_{\gamma} \theta_{\gamma'} - \theta_{\alpha} \theta_{\beta}.
\end{align*}
\end{proof}

\begin{lemma}
\label{rootserre}
Let $ \alpha' = \alpha_i + \alpha_{i+1} + \cdots + \alpha_{j} $ and 
$ \alpha'' = \alpha_{j+1} + \alpha_{j+2} + \cdots + \alpha_k $ such that $ p(\alpha')=0 $.  Then 
\begin{equation*}
\theta_{\alpha'} \theta_{\alpha''} \theta_{\alpha'} - \frac{\theta_{\alpha'}^2 \theta_{\alpha''}}{[2]} - \frac{\theta_{\alpha''} \theta_{\alpha'}^2 }{[2]} \in \mathcal{I}.
\end{equation*}
\end{lemma}

\begin{proof}
This is a straightforward application of Propositions ~\ref{pbw2}, ~\ref{pbw4}, and ~\ref{pbw5}.
\end{proof}

\begin{prop}
\label{pbw7}
If $ p(\alpha) = 1 $ then $ \theta_{\alpha}^2 \in \mathcal{I} $.
\end{prop}

\begin{proof}
This follows from induction using Lemma ~\ref{rootserre} where the base case is $ \theta_m^2 \in \mathcal{I} $.
\end{proof}

\begin{prop}
\label{pbwspan}
The algebra $ {\bf f} = {\bf f}' / \mathcal{I} $ is spanned by elements of the form
$$ \prod_{\alpha \in \Delta^+}^{<} \theta_{\alpha}^{n_{\alpha}}, $$
where the terms in the product are arranged by the ordering on $ \Delta^+ $, and $ n_{\alpha} \in \N $ for all $ \alpha$ with $ p(\alpha) = 0 $ and $ n_{\alpha} \in \lbrace 0, 1 \rbrace $ for all $ \alpha $ with $ p(\alpha)=1$.
\end{prop}

\begin{proof}
This follows directly from Propositions ~\ref{pbw1}, ~\ref{pbw2}, ~\ref{pbw3}, ~\ref{pbw4}, ~\ref{pbw5}, ~\ref{pbw6}, and ~\ref{pbw7}.
\end{proof}

\begin{prop}
\label{comultpbw}
$$ \Delta(\theta_{\alpha_i+\cdots+\alpha_j})= 
\theta_{\alpha_i+\cdots+\alpha_j} \otimes 1 + \sum_{r=i}^{j-1} (q^{-r \bullet (r+1)}-q^{r \bullet (r+1)}) \theta_{\alpha_{r+1}+\cdots+\alpha_j} \otimes \theta_{\alpha_i+\cdots+\alpha_r}
+ 1 \otimes \theta_{\alpha_i+\cdots+\alpha_j}.$$

\end{prop}

\begin{proof}
This is an easy induction on $j$.
\end{proof}

\begin{prop}
\label{formroot}
Let $ \alpha = \alpha_i + \cdots + \alpha_j$.  Then
$$ (\theta_{\alpha}, \theta_{\alpha}) = \begin{cases}
\frac{q^{-2(j-i)}}{1-q^2} & \textrm{ if } i \leq j < m, \\
q^{-2(j-i)} & \textrm{ if } i \leq j = m, \\
\frac{q^{2(j-i)}}{1-q^{-2}} & \textrm{ if } m < i \leq j, \\
q^{2(j-i)} & \textrm{ if } m=i \leq j,\\
q^{2(i+j-2m)} & \textrm{ if } i < m < j.
\end{cases} $$
\end{prop}

\begin{proof}
For each case, this is an easy induction on the length of the root using Propositions ~\ref{formproperties} and ~\ref{comultpbw}.
\end{proof}

\begin{prop}
\label{formprodroots}
The form on spanning elements is given by:
$$ (\prod_{\alpha \in \Delta^+}^{<} \theta_{\alpha}^{m_{\alpha}}, \prod_{\alpha \in \Delta^+}^{<} \theta_{\alpha}^{n_{\alpha}}) =
\prod_{\alpha \in \Delta^+} \delta_{m_{\alpha}, n_{\alpha}}(\theta_{\alpha},\theta_{\alpha})^{n_{\alpha}}(\sum_{r=0}^{n_{\alpha}-1} (-1)^{rp(\alpha)} q^{-r \alpha \bullet \alpha})(\sum_{r=0}^{n_{\alpha}-2} (-1)^{rp(\alpha)} q^{-r \alpha \bullet \alpha}) \cdots (1). $$
\end{prop}

\begin{proof}
This is proved using Proposition ~\ref{comultpbw} and the prescribed ordering of the terms in the product.  For more details see ~\cite[Lemma 10.2.1]{Yam}.
\end{proof}

\begin{defined}
\label{idealJ}
Let $ \mathcal{J} $ be the two sided ideal of $ {\bf f}' $ generated by elements
\begin{enumerate}
\item $ (q+q^{-1})\theta_{m} \theta_{m-1} \theta_{m+1} \theta_m  
- \theta_{m-1} \theta_{m} \theta_{m+1} \theta_{m} 
- \theta_{m} \theta_{m-1} \theta_{m} \theta_{m+1}
- \theta_{m} \theta_{m+1} \theta_{m} \theta_{m-1}
- \theta_{m+1} \theta_{m} \theta_{m-1} \theta_{m} $,
\item  $ \theta_m^2 = 0 $,
\item $ \theta_i \theta_j - \theta_j \theta_i  $ where  $ |i-j|>1$,
\item 
$ (q+q^{-1})\theta_i \theta_{i \pm 1} \theta_i - \theta_i^2 \theta_{i \pm 1} - \theta_{i \pm 1} \theta_i^2  $ where $ i \neq m $ and $ i \pm 1 \in I $.
\end{enumerate}
\end{defined}

We note again that the following theorem was proved by Yamane ~\cite{Yam}.

\begin{theorem}
\label{pbwbasis}
\begin{enumerate}
\item The algebra $ {\bf f} = {\bf f}' / \mathcal{I} $ has a basis with elements of the form
$$ \prod_{\alpha \in \Delta^+}^{<} \theta_{\alpha}^{n_{\alpha}}, $$
where the terms in the product are arranged by the ordering on $ \Delta^+ $, and $ n_{\alpha} \in \N $ for all $ \alpha$ with $ p(\alpha) = 0 $ and $ n_{\alpha} \in \lbrace 0, 1 \rbrace $ for all $ \alpha $ with $ p(\alpha)=1$
\item There is an equality of two-sided ideals: $ \mathcal{I} = \mathcal{J}$.
\end{enumerate}
\end{theorem}

\begin{proof}
By Proposition ~\ref{pbwspan}, these elements spans the algebra $ {\bf f}'$.  
By Proposition ~\ref{formprodroots}, these elements are orthogonal with respect to the form $ ( , ) $ so they are linearly independent and constitute a basis, thus proving the first part.

By Propositions ~\ref{gen1}, ~\ref{gen2}, ~\ref{gen3}, and ~\ref{gen4}, $ \mathcal{J} \subset \mathcal{I} $ and so there is a surjective algebra homomorphism
$ \mu \colon {\bf f}' / \mathcal{J} \rightarrow {\bf f}' / \mathcal{I}$.  This surjection is a bijection by the first part.
\end{proof}

Let $ \mathcal{A}=\mathbb{Z}[q,q^{-1}] $ and $ {}_{\mathcal{A}} {\bf f} $ to be the $ \mathbf{Z}[q, q^{-1}] $-subalgebra of $ {\bf f} $ generated by the divided powers $ \theta_i^{(n)} $.

\subsection{Graphical interpretation of $ ( , )$.}

Here we give a graphical interpretation of the inner product 
$ (\theta_{i_1} \cdots \theta_{i_b}, \theta_{j_1} \cdots \theta_{j_b})$.
Consider $ b $ marked points $ (1,0), (2,0), \ldots, (b,0) \in \R \times \lbrace 0 \rbrace$ colored by $ i_1, \ldots, i_b \in I $ respectively.
Consider $ b $ marked points $ (1,1), (2,1), \ldots, (b,1) \in \R \times \lbrace 1 \rbrace$ colored by $ j_1, \ldots, j_b \in I $ respectively.
Choose an immersion of $b$
strands into $\R\times [0,1]$ with these $2b$ points as the endpoints such that the labels at the endpoints of each strand agree.
each strand and label it by an element of $I$.  
We consider immersions modulo boundary--preserving
homotopies and call them pairings between sequences ${\bf i}=(i_1, \ldots, i_b) $ and $ {\bf j}=(j_1, \ldots, j_b)$.
Denote by $c'({\bf i},{\bf j})$ the set of all such pairings.

A minimal diagram $D$ of a pairing is a generic immersion that
realizes the pairing such that strands have no self--intersections and any two
strands intersect at most once.  We consider minimal diagrams up to
boundary--preserving isotopies.

An example of a minimal and a non-minimal pairing is given in ~\eqref{minnonmin}.
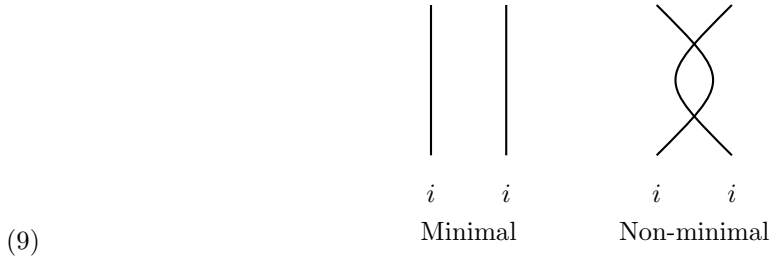
\begin{equation}
\label{minnonmin}
\begin{tikzpicture}
\draw (0,0) -- (0,2)[thick];
\draw (1,0) -- (1,2)[thick];
\draw (0, -.5) node{$i$};
\draw (1, -.5) node{$i$};
\draw (.5, -1) node{Minimal};

\draw (3,0) .. controls (4,1) .. (3,2)[][thick];
\draw (4,0) .. controls (3,1) .. (4,2)[][thick];
\draw (3, -.5) node{$i$};
\draw (4, -.5) node{$i$};
\draw (3.5, -1) node{Non-minimal};
\end{tikzpicture}
\end{equation}

We define the bidegree $ deg(C)=(deg_1(C), deg_2(C)) $ of the elementary diagrams by
\begin{equation}
\label{formdegrees}
\begin{tikzpicture}
\draw (-1,0) -- (7,0)[][very thick];
\draw (0,.5) node{bidegree};
\draw (0, -.5) node{diagram};
\draw (1,1) -- (1,-2)[][very thick];
\draw (1.5,.5) node{$(0,0)$};
\draw (1.5,-.5) -- (1.5,-1.5)[thick];
\draw (1.5, -1.75) node{$i$};
\draw (2,1) -- (2, -2)[very thick];
\draw (3,-.5) -- (3,-1.5)[thick];
\draw (3, -1.75) node{$ i$};
\draw (4,1) -- (4, -2)[very thick];
\draw (3, .5) node{$(0,i \bullet i)$};
\filldraw[black](3,-1) circle (2pt);
\draw (4.5,-1.5) -- (5.5,-.5)[thick];
\draw (5.5, -1.5) -- (4.5, -.5)[thick];
\draw (4.5,-1.75) node{$i$};
\draw (5.5, -1.75) node{$j$};
\draw (5.5,.5) node{$(-\delta_{i,m} \delta_{j,m},-i \bullet j)$};
\draw (-1,1) -- (-1,-2)[very thick];
\draw (7,1) -- (7,-2)[very thick];
\draw (-1,1) -- (7,1)[][very thick];
\draw (-1,-2) -- (7,-2)[][very thick];
\end{tikzpicture}
\end{equation}


For each pairing of $ {\bf i} $ and $ {\bf j}$, we choose one minimal such pairing.  Denote the set of all minimal pairings between these two sequences by
$ c({\bf i}, {\bf j})$.  If $ D \in c({\bf i}, {\bf j}) $ is a concatenation $ D = D_l \circ \cdots \circ D_1 $ where each $ D_i $ is an elementary crossing as in ~\eqref{formdegrees},
set $ deg(D) = (deg_1(D_l) \cdots deg_1(D_1), deg_2(D_l) \cdots deg_2(D_1))$.

Now set
$$ (\theta_{i_1} \cdots \theta_{i_b}, \theta_{j_1} \cdots \theta_{j_b})' =   \sum_{D \in c({\bf i}, {\bf j})} (-1)^{deg_1(D)} q^{deg_2(D)} \prod_{\gamma=1, i_{\gamma} \neq m}^b \frac{1}{1-q^{i_\gamma \bullet i_{\gamma}}}. $$
If the indexing set of the summation is empty, we take the sum to be zero.

\begin{prop}
There is an equality of bilinear forms $ (,) = (,)'$.
\end{prop}

\begin{proof}
This is an easy modification of ~\cite[Theorem 2.2]{Rei}.
It is proved directly by using Propositions ~\ref{formproperties} and ~\ref{formdividedpowers} and ~\eqref{twistedbialgebra}.
\end{proof}

Consider for example the quantity $ (\theta_m \theta_m, \theta_m \theta_m)$.  We saw earlier that this evaluates to zero.  We must consider all minimal pairings between $ (m, m) $ and $ (m,m) $ which are given in 
~\eqref{mmexample}.

\begin{equation}
\label{mmexample}
\begin{tikzpicture}
\draw (0,0) -- (0,1)[thick];
\draw (1,0) -- (1,1)[thick];
\draw (0, -.5) node{$m$};
\draw (1, -.5) node{$m$};
\draw (.5, -1) node{$D_1$};

\draw (3,0) -- (4,1)[thick];
\draw (4,0) -- (3,1)[thick];
\draw (3, -.5) node{$m$};
\draw (4, -.5) node{$m$};
\draw (3.5, -1) node{$D_2$};
\end{tikzpicture}
\end{equation}

Since $ deg(D_1) = (0,0) $ and $ deg(D_2) = (-1,0)$,
$ (\theta_m \theta_m, \theta_m \theta_m)' = (-1)^0 q^0 + (-1)^{-1} q^0 = 0$.


\section{Negative finite-dimensional dg algebras and their bigraded 
counterparts} 
\label{dgalgebras}


\subsection{Finite-dimensional algebras.} 
We work over  a fixed field $\Bbbk$. The Jacobson radical $J(A)$ of 
a $\Bbbk$-algebra $A$ is defined as the intersection of all maximal left 
ideals of $A$. It is a 2-sided ideal, being equal to the intersection of all maximal right ideals 
of $A$. If $A$ is, in addition, finite-dimensional over $\Bbbk$, the 
Jacobson radical is the unique maximal nilpotent 2-sided ideal of $A$ and 
the quotient ring $A/J(A)$ is semisimple~\cite{Ben}, being the product of matrix 
algebras with coefficients in division algebras $D_i$ over $\Bbbk$: 
$$ A/J(A) \ \cong \ \prod_{i=1}^m \mathrm{Mat}(n_i, D_i) .$$ 
Assume from now on that $A$ is finite-dimensional. Then $A$ has finitely-many 
isomorphism classes of simple left modules $L_1, \dots, L_m$, in bijection with 
terms in the above decomposition. Each $L_i \cong D_i^{n_i}$ can be realized 
as the column vector module over $ \mathrm{Mat}(n_i, D_i)$, one of the 
terms in the above product, with the remaining terms acting trivially. 
For each $L_i$ there exists a unique (up to isomorphism) indecomposable projective 
module $P_i$ with the property 
$$ \Hom_A(P_i, L_j) \cong \begin{cases} D_i^{op} & \text{if} \ i=j, \\  
 0 & \text{otherwise}.\end{cases}   $$
As a left module over itself, 
$$A\cong \oplusoop{i=1}{m} P_i^{n_i}.$$ 
Any finite-dimensional projective left $A$-module $P$ is isomorphic to the direct 
sum of $P_i$ with some multiplicities, 
 $$P \cong \oplusoop{i=1}{m} P_i^{k_i}.$$ 
The multiplicities $k_i$ are invariants of $P$. 

The Grothendieck group $K_0(A)$ of the category of finitely-generated projective left $A$-modules 
is free abelian of rank $m$ with generators $[P_1], \dots, [P_m]$. 
The Grothendieck group $G_0(A)$ of the category of finite-length $A$-modules 
is generated by $[M]$, over all finite-length $A$-modules $M$, with defining 
relations $[M_2]=[M_1]+[M_3]$ for each short exact sequence 
$$ 0 \lra M_1 \lra M_2 \lra M_3 \lra 0.$$ 
This group is also free abelian of rank $m$ with generators $[L_1], \dots, [L_m]$ 
being symbols of simple $A$-modules. There exists a bilinear pairing 
\begin{equation} 
K_0(A) \otimes_{\Z} G_0(A) \lra \Z, \ \  ([P], [M]) \ := \ \dim_{\Bbbk} \Hom_A(P,M), 
\end{equation} 
If each simple $L_i$ is absolutely simple, that is, $\mathrm{End}_A(L_i)= \Bbbk$ as a 
$\Bbbk$-algebra for all $i$, this pairing is perfect, and the bases $\{[P_1], \dots, [P_m]\}$ and 
$\{ [L_1], \dots, [L_m]\}$ of $K_0(A)$ and $G_0(A)$, respectively, are dual. 
For any finite-dimensional $A$, the pairing becomes perfect if we work over $\Q$, 
by tensoring $K_0(A)$ and $G_0(A)$ with rationals.  

Notice that $[P]$, for a finitely-generated projective $P$, can potentially have double 
meaning: as an element of $K_0(A)$ and as an element of $G_0(A)$. Usually, we 
view it as belonging to $K_0(A)$. The obvious homomorphism $K_0(A) \lra G_0(A)$ 
is, in general, neither injective nor surjective. It's an isomorphism if $A$ has finite 
homological dimension. 


\vspace{0.2in} 

\subsection{Graded algebras} The following question can serve to motivate this section: 
What kind of graded algebras have representation theory resembling that of finite-dimensional 
algebras? Here's a possible answer. We call a 
$\Z$-graded $\Bbbk$-algebra $A= \oplusop{i\in \Z} A^i$ \emph{gradual} if   
\begin{itemize} 
\item $A$ is bounded below: $A^i=0$ for $i\le N$, for some $N\in \Z$;  
\item $A$ is finite-dimensional in each degree: $\dim(A^i)< \infty$, for all $i\in \Z$; 
\item the center $Z(A)$ of $A$ is graded Noetherian and $A$ is finitely-generated as a 
graded module over $Z(A)$. 
\end{itemize} 
Note that $Z(A)$ is a graded algebra. The last condition is equivalent to the existence of 
a graded Noetherian subalgebra $B\subset Z(A)$ such that $A$ is finitely-generated as a 
graded $B$-module.  A gradual algebra is both left and right graded Noetherian.

For a graded $A$-module $M$ let $ M^a$ be the subspace in degree $a$.  Define the shifted module $A$-module $ M \{ r \} $ by 
$ (M \{ r \})^a = M^{a-r}$.

We define the graded Jacobson radical $J(A)$ of a graded ring $A$ as the intersection 
of all maximal left graded ideals.   Similar to the ungraded case, $ J(A) $ may also be defined as the set of elements of $ A $
annihilating all simple graded $ A $-modules (see for example ~\cite[Definition 11]{Ber}).

\begin{lemma}
\label{centerradical}
If $ A $ is gradual, then the intersection $Z(A)\cap J(A)$ contains the graded Jacobson radical of $Z(A)$. 
\end{lemma}  

\begin{proof}
Since $ A $ is finitely generated over its center, the hypotheses of (an obvious graded version of) ~\cite[Proposition 5.7]{Lam} hold giving that 
the Jacobson radical of $ Z(A) $ is contained in $ J(A) $.  It follows immediately that $ Z(A) \cap J(A) $ contains the Jacobson radical of $ Z(A)$.

\end{proof}

\begin{lemma}
\label{finitecodim}
Assume that $A$ is gradual. Then $J(A)$ has 
finite codimension in $A$, so that $A/J(A)$ is a finite-dimensional semisimple graded $\Bbbk$-algebra. 
\end{lemma}

\begin{proof}
By Lemma ~\ref{centerradical}, $ J(Z(A)) \subset J(A)$.  
Thus there is a surjection $ A/J(Z(A)) \twoheadrightarrow A/J(A) $ and so it suffices to prove that $ A/J(Z(A)) $ is finite-dimensional.
Since $ A $ is a gradual algebra, it is finitely generated over $ Z(A) $.  Therefore it is enough to show that $ Z(A)/J(Z(A)) $ is finite-dimensional.

Since $ A $ is a gradual, $ A^i =0$ for $ i \leq N $ for some $ N \in \mathbb{Z} $ and likewise for $ Z(A)$.
Any negative degree element of the commutative algebra $Z(A) $ is nilpotent, hence belongs to $ J(Z(A)) $.
Therefore, any simple graded $Z(A)$-module is concentrated in one degree, so that all positive degree elements of $Z(A) $ are in $J(Z(A))$ as well.
Thus $ Z(A)/J(Z(A)) $ is a quotient of $ Z(A)^0 $ and finite-dimensional.





\end{proof}



Furthermore, $A/J(A)$ is isomorphic to the product of graded matrix algebras, 
\begin{equation*} 
A/J(A) \ \cong \  \prod_{i=1}^m \mathrm{Mat}(f_i(q), D_i), 
\end{equation*}
with coefficients in division rings $D_i$ (division rings $D_i$ are necessarily concentrated 
in the zero degree). Here 
$f_i(q) \in \N[q,q^{-1}]$ are Laurent polynomials in $q$ with nonnegative integer 
coefficients, and $\mathrm{Mat}(f_i(q), D_i)$ is the endomorphism ring of 
the graded right $D_i$-vector space $D_i^{f_i(q)}$.  For more details see ~\cite{NVO}.

\begin{remark}
Lemma ~\ref{centerradical} could be strengthened to an equality of ideals using Lemma ~\ref{finitecodim}.
\end{remark}

Up to grading shifts there are finitely-many isomorphism classes of graded 
simple left $A$-modules. We can choose representatives $L_1, \dots , L_m$ for 
these classes with the property that zero is the lowest nontrivial degree 
of each $L_i$. Up to a grading shift, 
$L_i \cong D_i^{f_i(q)}$, with the action of $A$ factoring through $A/J(A)$. 
Any graded simple left $A$-module is finite-dimensional 
over $\Bbbk$ and is isomorphic 
to $L_i\{j\}$ for unique $ j \in \Z$ and $1\le i\le m$. Graded modules $L_i$ remain simple 
when viewed as modules without the grading. 

Let $A\mbox{-flgmod}$ be the abelian category of graded left $A$-modules of 
finite length. The Grothendieck group $G_0(A)$ is defined as the Grothendieck 
group of the category $A\mbox{-flgmod}$. By the Jordan-H\"older theorem, 
it's a free $\Z[q,q^{-1}]$-module with basis $\{[L_1], \dots, [L_m]\}$. 

The Grothendieck $\Z[q,q^{-1}]$-module  $K_0(A)$ of a graded ring 
$A$ has generators $[P]$, over finitely-generated 
graded projectives $P$ over $A$, defining relations $[P]=[P']+[P'']$ whenever 
$P\cong P'\oplus P''$ and $[P\{1\}]=q[P]$. 

Consider the natural pairing 
\begin{equation}\label{pairing2} 
K_0(A) \otimes_{\Z[q,q^{-1}]} G_0(A) \lra \Z[q,q^{-1}], \ \  
  ([P], [M]) = \mathrm{gdim} \ \HOM_A(P,M) .
\end{equation} 
Here $\HOM_A(P,M) = \oplusop{i\in \Z}\Hom_A(P\{i\}, M)$ is the 
graded vector space which is the sum of spaces of degree zero homomorphisms 
from various shifts of the first module to the second module. The pairing (\ref{pairing2}) 
is perfect if each simple $L_1, \dots, L_m$ is absolutely irreducible, and becomes 
perfect over $\Q[q,q^{-1}]$ for any gradual algebra $A$. 


\vspace{0.2in} 

\subsection{Dg algebras and modules over them} 
\label{dgbasics}

A dg algebra (differential graded algebra) over $\Bbbk$ is a graded $\Bbbk$-algebra 
$A=\oplusop{i\in \Z}A^i$ equipped 
with a $\Bbbk$-linear map $d$ of degree $1$ (differential) that satisfies 
$ d(ab) = d(a) b + (-1)^{|a|} a d(b) $ for all homogeneous $a,b\in A$, where $|a|$ is 
the degree of $a$. We often denote this dg algebra by $ (A,d)$.
If $ (A,d_A) $ and $ (B,d_B) $ are dg algebra, then $ (A \otimes B, d) $ is a dg algebra
where $ d(a \otimes b) = d_A a \otimes b + (-1)^{|a|} a \otimes d_B b $ for homogeneous elements
$ a \in A $ and $ b \in B$.

A left module $M$ over a dg algebra is a graded left module over $A$ equipped 
with a $\Bbbk$-linear map $d_M$ of degree $1$ such that 
$d_M(am) = d(a)m + (-1)^{|a|}ad_M(m)$ for all homogeneous $a\in A, m \in M$. 
Define the shifted module $M[r]$ to have the same underlying vector space as $M$ but
$ (M[r])^a = M^{a+r}$.  The action of the differential is given by $ d_{M[r]} = (-1)^r d_{M}$.
The action of $A$ is given by $ a \circ m = (-1)^{r|a|} am $ where $ am $ is the normal action of $a$ on $ m \in M$.

If $ (A,d) $ is a dg subalgebra of the dg algebra $ (B,d) $ and $ M $ is a dg module over $ A$, then
$ \text{Ind}_A^B M = B \otimes_A M $ is a dg module over $ B $.

Let $\Lambda_d= \Bbbk[d]/(d^2)$ be the graded exterior algebra on one generator $d$ 
of degree $1$. Define $A_d = A \rtimes \Lambda_d $ as the cross-product of 
$A$ and $\Lambda_d$ such that 
\begin{eqnarray*} 
(a \otimes 1)(b \otimes x) & = &  ab \otimes x, \\
(a \otimes d)(b \otimes x) & = & a d(b) \otimes x + (-1)^{|b|} ab \otimes dx, 
\end{eqnarray*} 
for homogeneous $a,b\in A,  x\in \Lambda_d$. A left dg $A$-module $M$ is the same as a left 
graded $A_d$-module. 

An idempotent $ e \in A $ such that $ e=dx$ for some $ x \in A $ is said to be contractible.  
If $e$ is not in the image of $ d $ then it is non-contractible.  
A dg module $M$ is defined to be homotopically trivial if there is a homotopy equivalence between $ M $ and the zero module.
If $ A $ is a dg algebra and $ e \in A $ is an idempotent such that $ e=dx$ for some $x \in A $, then 
the dg module $ Ae $ is homotopically trivial (see ~\cite[Lemma 4]{Kh1}).
A map $ f \colon M \rightarrow N $ of dg modules is said to be a quasi-isomorphism if it induces an isomorphism on cohomology. 
A dg module $M$ is acyclic if it is quasi-isomorphic to the zero module.



\vspace{0.2in} 

\subsection{Negative finite-dimensional  dg algebras} 
\label{negfddg}

Let $A^{\le i}= \oplusop{j\le i} A^j$. Following Keller~\cite[Summary]{Ke1}, we say that $A$ is a 
negative dg algebra if 
$A= A^{\le 0}$. A negative dg algebra which is finite-dimensional over the ground 
field $\Bbbk$ will be called a negative finite-dimensional (nfd) dg algebra. 
If $A$ is nfd and $N\in \N$ is the largest number such that $A^{-N}\not= 0$, then 
$A_d$ is concentrated in degrees $-N$ through $0$. 

Recall that the graded Jacobson radical of a graded finite-dimensional algebra coincides with the   
Jacobson radical of the same algebra without the grading~\cite{NVO}. 
Viewing a nfd dg algebra $A$ as a graded algebra, $J(A) = A^{\le -1}\oplus J(A^0)$. 
We would like to understand the Jacobson radical of $A_d$. Let 
$$ J_{\bullet}(A) \ = \ A^{\le -2} \oplus \tilde{A}\oplus J(A^0), \ \ \ \ 
  \tilde{A} \ := \ \{ a\in A^{-1}| \ d(a) \in J(A^0)\} .$$ 
It is not hard to see that $J_{\bullet}(A)\subset J(A)$ is a $d$-stable 2-sided graded ideal of $A$.  
Consequently, $J_{\bullet}(A) \rtimes \Lambda_d$ is a 2-sided graded ideal 
of $A_d$ and 
$$ A_d/ (J_{\bullet}(A) \rtimes \Lambda_d) \cong (A_{\bullet})_d, \ \   \mathrm{where} \ \   
 A_{\bullet} \ := \ A/J_{\bullet}(A).$$ 
Denote $(A_{\bullet})_d$ simply by $A_{\bullet d}$. The 2-sided ideal 
$J_{\bullet}(A) \rtimes \Lambda_d$ is nilpotent (if $J(A_0)^n=0$, then 
the product of any $n(N+1)$ elements of $J(A)$ is zero, and the product 
of any $n(N+1)+2$ elements of $J_{\bullet}(A) \rtimes \Lambda_d$ is zero), therefore it  
belongs to the Jacobson radical $J(A_d)$. The quotient map 
$A \lra A_{\bullet}$ induces a quotient map $A_d \lra A_{\bullet d}$. The latter 
induces an isomorphism of Grothendieck groups $G_0(A_d) \cong G_0(A_{\bullet d})$. 
The action of $A_d$ on any simple $A_d$-module factors through the action 
of $A_{\bullet d}$. Thus, to understand (graded) simple $A_d$-modules, it's enough 
to understand (graded) simple $A_{\bullet d}$-modules. Dg algebra $A_{\bullet}$ has 
the form 
$$ 0 \lra A^{-1}_{\bullet} \stackrel{d}{\lra} A^0_{\bullet} \lra 0, $$ 
with $d$ being an injective map. $d(A^{-1}_{\bullet})$ is a
 2-sided ideal of $A^0_{\bullet}$. The algebra $A^0_{\bullet}\cong A^0/J(A^0)$
is semisimple, and any 2-sided ideal in it has a complement, being a product of several 
terms in the unique decomposition of it into matrix algebras. 
We can reshuffle the terms so that 
\begin{eqnarray*} 
& & A^0_{\bullet} \cong \prod_{i=1}^{m_I+m_{II}} \mathrm{Mat}(n_i, D_i) , \\
 & & d(A^{-1}_{\bullet}) = \prod_{i=m_I+1}^{m_I+m_{II}} \mathrm{Mat}(n_i, D_i), 
\end{eqnarray*} 
for some $m_I,m_{II}$. In particular, $m_I+m_{II}$ is the number of isomorphism classes 
of simple $A^0$ (and $A$) modules.  Let 
\begin{eqnarray*} 
A_{I}  & = & \prod_{i=1}^{m_I} \mathrm{Mat}(n_i, D_i) , \\ 
A_{II} & = & \prod_{i=m_I+1}^{m_I+m_{II}} 
\mathrm{Mat}(n_i, D_i) \ = \ d(A_{\bullet}^{-1}). 
\end{eqnarray*} 
Then $A^0_{\bullet} = A_I \times A_{II}$. Denote by $e_I, e_{II}$ the 
unit elements of $A_I, A_{II}$. These are idempotents in $A_{\bullet}$, with $e_{II}$ 
being the maximal contractible idempotent, and 
$1=e_I+e_{II}$. Let $y \in A^{-1}_{\bullet}$ be 
the unique element such that $d(y) = e_{II}\in A_{II}$. 
The dg algebra $A_{\bullet}$ decomposes: 
$$ A_{\bullet} \cong A_I \times ( A_{II}\rtimes \Lambda_y), $$ 
where $\Lambda_y = \Bbbk[y]/(y^2)$ is the exterior algebra in one generator 
of degree $-1$, $ A_{II}\rtimes \Lambda_y$ is the tensor product of two algebras, 
$d(A_I)=0, d(A_{II})=0$, and $d(e_{II}\otimes y)=e_{II}$. 

The algebra $A_{\bullet d}$ decomposes 
$$ A_{\bullet d} \cong (A_I \rtimes \Lambda_d) \times ((A_{II}\rtimes \Lambda_y) \rtimes  
  \Lambda_d) ,$$ 
where  
$$(A_{II}\rtimes \Lambda_y) \rtimes \Lambda_d \cong A_{II}\rtimes 
(\Lambda_y \rtimes \Lambda_d) $$
is isomorphic to the tensor product of $A_{II}$ with the cross-product algebra 
$$\Lambda_y\rtimes \Lambda_d \cong  \Bbbk\langle y, d\rangle /(dy  +yd -1 ).$$ 
This cross-product algebra is isomorphic to the algebra of $2\times 2$ matrices 
(take $y,d$ to the elementary matrices $E_{12}$ and $E_{21}$). As a graded algebra, 
$$\Lambda_y\rtimes \Lambda_d \cong \mathrm{Mat}(1+q,\Bbbk). $$ 
The Jacobson radical of $A_{\bullet d}$ is $A_I \otimes \Bbbk d $, and 
$$ A_d/J(A_d) \cong A_{\bullet d}/J(A_{\bullet d}) \cong A_I \times \mathrm{Mat}(1+q, 
  A_{II}).$$ 
This leads to a description of (graded) simple $A_d$-modules in terms of those for $A$. 

\begin{prop}
\label{simpledgmodules}
Let $L_1, \dots, L_{m_I+m_{II}}$ be graded simple $A$-modules, one for each isomorphism class 
up to grading shifts. In our notations, $L_i \cong D_i^{n_i}$. 

The (graded) simple $A_d$-modules, up to isomorphism and grading shifts, are of two types: 

(I) $A$-modules $L_i$, with $d$ acting trivially, $1\le i \le m_I$.  These modules live 
in one degree. 

(II) $\widehat{L}_i= L_i \oplus L_i[1]$ (as $A^0$-modules), $m_I+1\le i \le m_I+m_{II}$, 
with $d$ taking $L_i[1]$ isomorphically to $L_i$: 
  $$ 0 \lra L_i \cdot v_0 \lra L_i \cdot v_1 \lra 0.$$ 
Here $v_0, v_1$ are generators, and $d(v_0)=v_1$. The subspace $A^{\le -2}$ acts 
trivially  on $\hat{L}_i$, while $A^{-1}$ acts via the quotient 
map $A^{-1}\lra A_{\bullet}^{-1},$ 
with $y' v_1=v_0$ for any $y'$ which goes to $y$ under the quotient map.  
\end{prop}

We also have a description of the indecomposable summands of $ A$ as dg $A$-modules.

\begin{prop}
\label{projectivedgmodules}
The indecomposable summands of dg algebra $A$ viewed as left dg modules over itself up to isomorphism are of two types:

(I) Modules $ P_1, \ldots, P_{m_I} $ which are cohomologically non-trivial.

(II) Acyclic modules $ P_{m_I+1}, \ldots, P_{m_I + m_{II}} $.





\end{prop}

\begin{proof}
From above we determined that 
$$ A_d/J(A_d) \cong A_{\bullet d}/J(A_{\bullet d}) \cong A_I \times \mathrm{Mat}(1+q, A_{II}).$$ 
The graded summands (up to isomorphism) of this quotient algebra are in bijection with the indecomposable idempotents (up to conjugacy).  
Considering these objects as modules over the corresponding dg algebra, those summands coming from idempotents in the first factor (non-contractible idempotents) 
are homotopically non-trivial.  The idempotents of the second factor are contractible and give rise to homotopically trivial objects.
By ~\cite[Theorem 1.7.3]{Ben}, there is a one-to-one correspondence between conjugacy classes of indecomposable idempotents in $ A_d $ and in $ A_d/J(A_d)$.
Clearly non-contractible idempotents lift to non-contractible idempotents.
It remains to show that contractible idempotents lift to contractible idempotents.

Let $ \pi \colon R \rightarrow S $ be a surjective dg map of nfd fg algebras.
Call the differential for both algebras $ d$.
Suppose $ b \in S $ and $ b=d\beta$.
Let $ \pi(\alpha)=\beta$.  Then $\pi(d\alpha)=d\pi(\alpha)=d\beta=b$.  So a contractible element $ b $ has some contractible lift.

Suppose $ x=d\tilde{x} $ and $ y=d\tilde{y} $ are contractible elements.   Then $ xy $ is contractible since
$ xy = d\tilde{x}d\tilde{y}=d(\tilde{x}d\tilde{y})$.

Now the proof of \cite[Theorem 1.7.3]{Ben} along with the above two facts give that contractible idempotents lift to contractible idempotents.
\end{proof}

\vspace{0.1in} 

As before, we're assuming that $A$ is a nfd dga. For $M\in D(A)$ define 
$$ wt_i(M) \ = \ \sum_{j\in \Z}\dim_{D_i} \left(\mathrm{Ext}^j(M,L_i)\right) .$$ 
Since $D^{op}_i = \mathrm{End}_A(L_i)$, each $\mathrm{Ext}^j(M,L_i)$ is a 
(left) vector space over the division ring $D_i$, with a well-defined dimension. 
We see that $wt_i(M)$, which may or may 
not be finite, is a sort of "multiplicity" of $L_i$ in $M$. Let 
$$ wt(M) \ = \ \sum_{i=1}^{m_I} wt_i(M).$$ 
It is easy to see that  $wt(M) < \infty$ if $M$ is compact. 

\begin{prop}
\label{projfilt}
Let $ A $ be a nfd dg algebra.  
Then any compact object of the derived category $D(A)$ 
is quasi-isomorphic to a finite filtered complex built out of shifts of $P_i$'s, 
for $1\le i \le m_I$. 
\end{prop}

\begin{proof}
Start with $M\in D^c(A)$. Since $M$ is compact and $A$ is finite-dimensional, the total 
cohomology $H(M)$ is a
 finite length $H(A)$-module (equivalently, $H(M)$ is a finite-dimensional $\Bbbk$-vector 
space). If $H(M)=0$ then $M$ is quasi-isomorphic to the zero complex and $wt(M)=0$. 
Assuming otherwise, choose largest $j$ such that $H^j(M)\not= 0$. The $H^0(A)$-module 
$H^j(M)$ surjects onto $L_i$ for some $i$. We use this surjection to set up a morphism 
$\alpha: M \lra L_i[-j]$ in the derived category 
$$ M  \longleftarrow M'\stackrel{\alpha'}{\lra} L_i [-j],$$
where $M' $ is quasi-isomorphic to $M$, given by 
$$ \dots \lra M^{j-2}\lra M^{j-1} \lra \mathrm{ker} d^j  \lra 0 \lra \dots , $$
and   $\alpha'$ comes from composite map $\mathrm{ker} d^j \lra H^j(M) \lra L_i$.  
Morphism $\alpha$ is nontrivial since it induces a nontrivial map on cohomology. 

We can write $P_i \cong Ae$ for some minimal idempotent $e\in A$ of degree $0$. Necessarily
$d(e)=0$ and $d(ae) = d(a)e$ for $a\in A$. 
The quotient map $P_i \lra L_i$ from indecomposable projective $P_i$ to its head 
takes $e$ to some $v_0\in L_i$. We can find $v\in M'$ such that $ev = v$ and $v$ goes to 
$v_0[-j]$ under the map $\alpha'$, in particular, $d(v)=0$. This results in a module homomorphism 
$P_i[-j] \cong Ae[-j] \lra M'$ taking $ae$ to $av$. Composing it with the inclusion $M'\subset M$ 
gets us a homomorphism $\beta: P_i[-j] \lra M$. This is a nontrivial homomorphism in $D(A)$ 
since it induces a nontrivial map on cohomology. Let $M_1=\mathrm{C}(\beta)$ be the 
cone of $\beta$. Notice that $M_1$ is a filtered dg $A$-module, with subquotients 
isomorphic to $P_i[1-j]$ and $M$, respectively. Moveover, since 
$$ \mathrm{Ext}^k(P_i, L_s) = \begin{cases}  D_i^{op} & \textrm{if} \  i=s, \ k=0, \\  
    0 & \textrm{otherwise} , \end{cases} $$ 
from the distinguished triangle 
 $$ P_i[-j] \lra M \lra M_1 \stackrel{[1]}{\lra} $$ 
we derive that 
$$ wt_r(M_1) =\begin{cases} wt_i(M)-1 & \textrm{if} \ i=r , \\
                         wt_r(M) & \textrm{otherwise}. \end{cases} $$ 
Hence, $wt(M_1) = wt(M)-1$. Now repeat the above procedure with $M_1$. 
We obtain a filtered dg $A$-module $M_2$ with $wt(M_2) = wt(M)-2$ and 
subquotients isomorphic to $M_1$ and $P_{i'}[1-j']$ for some $i', j'$. 
Let $n=wt(M)$. After $n$ steps we obtain a filtered module $M_n$. The lowest 
term in the filtration is $M$ and the module $M_n/M$ has a filtration with 
terms isomorphic to $P_i[1-j]$ for various $i,j$, with the number of terms 
isomorphic to $P_i[1-j]$ equal to the dimension of the ext group 
$ \mathrm{Ext}^j(M, L_i) $ as a $D^{op}_i$-vector space. The map 
$(M_n/M) [-1] \lra M $, encoded in the filtration, is a quasi-isomorphism, since 
$wt(M_n)=0$ and $M_n$ has trivial cohomology.  
\end{proof}

Let $\mc{P}_c(A)$ be the category whose objects are left dg $A$-modules that 
admit a finite filtration with subsequent quotients isomorphic to $P_i[j]$, for 
$1\le i \le m_I$ and $j\in \Z$ and morphisms are homomorphisms of 
dg modules up to chain homotopies. Then the natural functor 
$\mc{P}_c(A) \lra D^c(A)$, which is the identity on objects and morphisms, is 
an equivalence of categories. Likewise, the inclusion $\mc{P}_c(A) \lra \mc{P}^c(A)$ 
of $\mc{P}_c(A)$ into the category of compact cofibrant dg modules is 
an equivalence. 

Define the Grothendieck group $K_0(A)$ of a dg ring $A$ to be the 
Grothendieck group of the triangulated category $D^c(A)$. 
When $A$ is a nfd dg algebra over $\Bbbk$, our results show that 
the set $\{ [P_1], \dots, [P_{m_I}]\}$ spans abelian group $K_0(A)$. 
At the same time, each $L_i$, $1\le i \le m_I$, defines a homomorphism 
$$ K_0(A) \lra \Z$$ 
taking $[M]$, for $M\in D^c(A)$, to 
$$ \sum_{j\in \Z} (-1)^j \dim_{D_i^{}} \mathrm{Ext}^j(M, L_i).$$ 
This homomorphism takes $[P_i]$ to $1$ and $[P_k]$ for $k\not= i$ to $0$. 
Hence, $[P_1], \dots, [P_{m_I}]$ are linearly independent in $K_0(A)$. 
This proves:

\begin{prop} For a nfd dg algebra $A$ the Grothendieck group $K_0(A)$ is 
free abelian with basis $\{ [P_1], [P_2], \dots, [P_{m_I}]\}$. 
\end{prop} 

\vspace{0.1in} 


Let $ D_{fl}(A) $ be the full subcategory of $ D(A) $ consisting of modules isomorphic to some finite length $A$-module.
Let $ D_{cfl}(A) $ be the full subcategory of $ D(A) $ consisting of modules $M$ such that the $H(M)$ has finite length as a module over $ H(A)$.
Since $H(A)$ is finite-dimensional, a module over $H(A)$ has finite length if and 
only if it is finite-dimensional. Therefore, we could also call $D_{cfl}(A)$ the category of cohomologically finite-dimensional dg $A$-modules.  

Using the techniques of Proposition ~\ref{projfilt} one could find a finite filtration of an object in $ D_{cfl}(A) $ whose successive quotients are quasi-isomorphic to shifted simple modules $L_i$ for $ i=1,\ldots,m_I$.
This proves:

\begin{prop} The inclusion functor $D_{fl}(A) \subset D_{cfl}(A)$ is 
an equivalence of categories. 
\end{prop} 


Define the Grothendieck group $G_0(A)$ as the Grothendieck group of the 
triangulated category $D_{cfl}(A)$ of cohomologically finite length modules.  

\begin{prop} Let $A$ be a nfd dg algebra. The Grothendieck group $G_0(A)$ 
is free abelian with basis $\{ [L_1], [L_2], \dots, [L_{m_I}]\}$. 
\end{prop}

The bifunctor 
$$ D^c(A) \times D_{cfl}(A) \lra D^c (\Bbbk) $$ 
that takes $P\times M$ to the complex $\mathrm{RHom}(P,M)$ descends 
to the bilinear form on Grothendieck groups
$$ K_0(A) \otimes G_0(A) \lra \Z , \ \  ([P], [M]) \ = \ \sum_{j\in \Z}(-1)^j \dim \mathrm{Ext}^j(P, M).  $$ 
This form is perfect if simple modules $L_1, \dots, L_{m_I}$ are absolutely irreducible. 

Summarizing, the minimal categories that we could work with are $\mc{P}_c(A)$ 
and $D_{fl}(A)$. Category $\mc{P}_c(A)$ is equivalent to $\mc{P}^c(A)$ 
and $D^c(A)$. Category $D_{fl}(A)$ is equivalent to $D_{cfl}(A)$.

Inclusions $\mc{P}_c(A)\subset D_{fl}(A)$ and 
 $\mc{P}^c(A) \subset D_{cfl}(A)$ are fully faithful. Each is an equivalence if and 
only if the other one is. If that's the case, 
we say that $A$ has finite cohomological dimension. It might be natural 
to define  cohomological dimension of a negative dg algebra $A$ as the largest $k$ such that 
$\mathrm{Ext}^k(M, M')\not= 0$ for some $M, M'$ concentrated in zero cohomological degree. 

\vspace{0.2in} 

\subsection{Negative dg gradual algebras} 

A bigraded dg algebra 
$$A=\oplusop{i,j\in \Z} A^{i,j}, \quad \quad d: A^{i,j}\lra A^{i+1,j},$$ 
is called a \emph{negative dg gradual algebra} (ndgg algebra, for short)
if 
\begin{itemize}
\item $A^{i,j}=0$ for $i> 0$ and any $j$,  
\item $A^{0,\ast}= \oplusop{j\in \Z} A^{0,j} $ is a gradual algebra, 
\item $A$ is finitely-generated as a bigraded module over $A^{0,\ast}$. 
\end{itemize} 
Notice that cohomological grading of a ndgg algebra is bounded from below:
there exists $N \in \N$ such that $A^{i,j}=0$ for all $i\le -N, j\in \Z$. 
The $q$-grading is bounded from below as well. 

A graded dg module $M$ over $A$ is a bigraded module $ M= \oplus_{i,j \in \mathbb{Z}} M^{i,j} $ 
equipped with a map $ d \colon M^{i,j} \rightarrow M^{i+1,j} $ for each $i $ and $j$ making $M$ a dg module over $A$.

One may form the shifted module $ M[r]\{s\} $ where $[r]$ means a shift in the first (homological) grading and $ \{s \} $ is a shift in the second (internal) grading.
Once again it is convenient to form the cross-product algebra $ A_d $ which is bigraded with $ d \in A_d^{1,0} $.

For any integer $r$ let 
\begin{equation*}
M^{\leq r, \ast} = \bigoplus_{i \leq r, j \in \mathbb{Z}} M^{i,j}, \hspace{.5in}
M^{\geq r, \ast} = \bigoplus_{i \geq r, j \in \mathbb{Z}} M^{i,j}, \hspace{.5in}
M^{r, \ast} = M^{\leq r, \ast} \cap M^{\geq r, \ast}.
\end{equation*}

\begin{prop}
\label{projngdg}
Let $ A $ be a negative dg gradual algebra.  
\begin{enumerate}
\item Let $L_1, \dots, L_{m_I+m_{II}}$ be bigraded simple $A$-modules, one for each isomorphism class 
up to grading shifts. 

The (bigraded) simple $A_d$-modules, up to isomorphism and grading shifts, are of two types: 

(I) $A$-modules $L_i$, with $d$ acting trivially, $1\le i \le m_I$. 

(II) $\widehat{L}_i= L_i \oplus L_i[1]$ (as $A^0$-modules), $m_I+1\le i \le m_I+m_{II}$, 
with $d$ taking $L_i[1]$ isomorphically to $L_i$.

\item The indecomposable summands of dg algebra $A$ viewed as left dg modules over itself up to isomorphism are of two types:

(I) Modules $ P_1, \ldots, P_{m_I} $ which are cohomologically non-trivial.

(II) Acyclic modules $ P_{m_I+1}, \ldots, P_{m_I + m_{II}} $.

\end{enumerate}

\end{prop}

\begin{proof} 
Since $ A^{0,*} $ is assumed to be a gradual algebra, it follows from Lemma ~\ref{finitecodim} that $ A^{0,*}/J(A^{0,*}) $ is finite-dimensional.
Since $A$ is a negative gradual dg algebra, it is finitely generated over $ A^{0,*} $.  Thus, 
$ A/J(A^{0,*}) $ is finite-dimensional.
As in Section ~\ref{negfddg}, $ J(A^{0,*}) \subset J_{\bullet}(A) $ so there is a surjection
$ A/J(A^{0,*}) \twoheadrightarrow A/J_{\bullet}(A) = A_{\bullet} $.
Thus $ A_{\bullet} $ and consequently $ A_{\bullet d} $ are finite-dimensional.
The proposition now follows as an easy graded version of Propositions ~\ref{simpledgmodules} and ~\ref{projectivedgmodules}.
\end{proof}

Define a graded version of the weight of a module $M$ which was introduced in Section ~\ref{negfddg} as
\begin{equation*}
wt(M)=\sum_{i=1}^{m_I} wt_i(M),
\end{equation*}
where
\begin{equation*}
wt_i(M)= \sum_{j \in \mathbb{Z}} \sum_{k \in \mathbb{Z}} \dim_{D_i}(\Ext^j(M,L_i \lbrace k \rbrace)).
\end{equation*}

\begin{prop}
\label{grothgroups}
Let $ A $ be a negative dg gradual algebra.  
\begin{enumerate}
\item $ K_0(A) $ is a free $ \mathbb{Z}[q,q^{-1}]$-module with basis $ \lbrace [P_1], [P_2], \ldots, [P_{m_I}] \rbrace $.
\item $ G_0(A) $ is a free $ \mathbb{Z}[q,q^{-1}]$-module with basis $ \lbrace [L_1], [L_2], \ldots, [L_{m_I}] \rbrace $.
\end{enumerate}
\end{prop}

\begin{proof}
We will only prove the first part; the second part is easier.
Let $M$ be an object of $D^c(A)$.  

Claim 1: $M$ has cohomology in only finitely many cohomological degrees.

Proof of Claim 1: Let $N$ be a positive integer such that $ A^{i,j}=0$ for all $i \leq -N, j \in \mathbb{Z}$.
Let $ R' $ and $R$ be two integers such that $ R-N \geq R'+2$.
Clearly $ dM^{\leq R', \ast} \oplus A M^{\geq R, \ast}$ is a dg submodule.
Let $ M_{\{R',R \}} = M/(dM^{\leq R', \ast} \oplus A M^{\geq R, \ast}) $.
The projection $ M \rightarrow M_{\{R',R \}} $ induces an isomorphism in cohomology in cohomological degrees
$ R'+1, \ldots, R-N-1 $.
If $M$ has cohomology in infinitely many degrees then there are infinitely many pairs $(R',R)$ giving non-trivial maps
$ M \rightarrow M_{\{R',R \}} $ on cohomology.  Then $ \Hom(M, \oplus_{(R',R)} M_{\{R',R \}}) $ is not isomorphic to $ \oplus_{(R',R)} \Hom(M, M_{\{R',R \}}) $ which contradicts the compactness of $M$.

Claim 2: Let $p$ be the largest integer such that $ H^p(M) \neq 0 $.  Then $H^p(M)$ is a finitely generated $ H^0(A) $-module.

Proof of Claim 2: Let $ \widehat{M}=M^{\leq p-1, \ast} \oplus ker(d_{|M^{p,\ast}})$.  This is a dg submodule.  Since $ H^i(M)=0$ for $ i >p$ the inclusion $ \widehat{M} \hookrightarrow M$ is a quasi-isomorphism.
Clearly $ H^p(M) $ is an $ H^0(A)$-module.  Since the differential acts trivially on $ A^0$, there is a surjection $ A^{0, \ast} \twoheadrightarrow H^0(A)$.  Thus $ H^p(M) $ is naturally an $ A^{0,\ast}$-module.
By having $ A^{<0,\ast} $ act trivially, $ H^p(M)$ becomes an $A$-module and a dg module with the trivial differential.  

There is a map $ \phi \colon \widehat{M} \rightarrow H^p(M) $ where $ M^{\leq p-1, \ast} $ gets mapped to zero and $ \phi $ projects $ ker(d_{|M^{p,\ast}}) $ onto $H^p(M)$.  
The map $ \phi $ clearly commutes with the differential and commutes with the action of $A$.  

$ H^p(M) $ as an $A$-module is compact since $ \widehat{M} $ which is quasi-isomorphic to $M$ is compact.  Since the differential acts trivially on the compact $A$-module $ H^p(M)$, it must be 
a finitely generated $A$-module.  Since $ A$ is finitely generated over $ A^{0, \ast}$ and $ A^{0, \ast}$ acts on $ H^p(M)$ through $ H^0(A)$, it follows that $H^p(M)$ must be a finitely generated $H^0(A)$-module.


Claim 3: $ \sum_i \sum_k \dim_{D_i} \Hom(H^p(M), L_i \{ k \}) < \infty $.

Proof of Claim 3: If the claim were not true then $ H^p(M) $ would not be finitely generated which would contradict Claim 2.

Claim 4: There is a finite filtration of $ M$ whose successive quotients are isomorphic in $D^c(A)$ to objects of the form $ P_i\{r \}[s]$ for $ i=1,\ldots,m_I$ and $ r,s \in \Z$.

Proof of Claim 4: 
Since $ A $ is a negative dg gradual algebra, $ A^{0, \ast} $ is a gradual algebra and hence Noetherian.  It follows that the quotient algebra $ H^0(A)$ is also Noetherian.  
Since $ M$ is compact, the $H^0(A)$-module $ H^p(M) $ maps onto $L_i \{k \} $ for some $i$ and $k$.
Using the techniques of Section ~\ref{negfddg} there is a distinguished triangle
\begin{equation*}
P_i[-p] \{k \} \rightarrow M \rightarrow M_1
\end{equation*}
where $ H^i(M_1)=0$ for $ i >p$ and $ wt(H^p(M_1))=wt(H^p(M))-1$.

Now we continue this procedure with object $M_1$.  By the above arguments this process will terminate after a finite number of steps.
Thus we get a finite filtration of $ M$ whose successive quotients are isomorphic to shifts (in both degrees) of objects of the form $ P_i$ for $ i=1,\ldots,m_I$.

The proposition now follows from Claim 4.





\end{proof}

The bifunctor 
$$ D^c(A) \times D_{cfl}(A) \lra D^{c} (\Bbbk) $$ 
that takes $P\times M$ to the complex $\mathrm{RHom}(P,M)$ descends 
to the bilinear form on Grothendieck groups
$$ K_0(A) \otimes_{\Z[q,q^{-1}]} G_0(A) \lra \Z[q,q^{-1}] , \ \  ([P], [M]) \ = \ \sum_{j\in \Z} \ \sum_{k\in \Z}(-1)^j \dim 
  \mathrm{Ext}^j(P, M \{k \})q^k.  $$ 
This form is perfect if simple modules $L_1, \dots, L_{m_I}$ are absolutely 
irreducible. 


\vspace{0.1in}

Generalizations of finite-dimensional algebras considered in this section 
fit into the following diagram. 

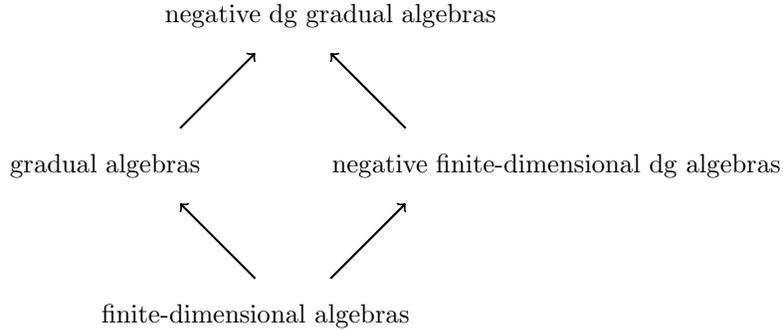
\begin{equation*}
\begin{tikzpicture}

\draw (-1,0) node{gradual algebras};
\draw (0,.5) -- (1,1.5)[->][thick];
\draw (2,2) node{negative dg gradual algebras};

\draw (0,-.5) -- (1, -1.5)[<-][thick];
\draw (1,-2) node{finite-dimensional algebras};

\draw (5,0) node{negative finite-dimensional dg algebras};
\draw (3,.5) -- (2,1.5)[->][thick];
\draw (3,-.5)--(2,-1.5)[<-][thick];

\end{tikzpicture}
\end{equation*}


We saw that many features of the basic representation theory of finite-dimensional 
algebra survive for these more general objects, including the structure 
of Grothendieck groups $K_0(A)$, $G_0(A)$, and of the pairing between them. 

\vspace{0.2in} 





\section{Dg algebras $R_{}(\nu)$} 
\label{R(nu)}
\subsection{The algebra $ R(\nu)$.}
Fix $ \nu = \sum_{i=1}^{m+n-1} \nu_i i$.  Let $ |\nu|=\sum_{i=1}^{m+n-1} \nu_i$.
Let $ \textrm{Seq}(\nu) $ be the set of all sequences $ {\bf i}=(i_{p_1}, \ldots, i_{p_{|\nu|}}) $ where $ i_{p_k} \in I $ for each $ k $ and $ i $ appears $ \nu_i $ times in the sequence.
We say that the $ k $-th entry of $ {\bf i} $ is bosonic if it is not equal to $ m $.  We say it is fermionic if it is equal to $ m $.
The symmetric group $ S_{|\nu|} $ acts on $ {\bf i} $ with the simple transposition $s_k$ exchanging the entries in positions $k$ and $k+1$.

Let $ R_{}(\nu) $ be the algebra over a field $ \Bbbk $ generated by $ y_k $ for $ k=1, \ldots, |\nu|$,
$ \psi_k $ for $ k=1, \ldots, |\nu|-1$, and $ e(\bf i) $ where $ {\bf i} \in \textrm{Seq}(\nu)$, with relations:
\begin{enumerate}
\item $ e({\bf i}) e({\bf j}) = \delta_{{\bf i}, {\bf j}} e({\bf i}) $
\item $ y_r e({\bf i}) = 0 $ if $ {\bf i}_r = m $
\item $ y_r e({\bf i}) = e({\bf i}) y_r $
\item $ \psi_k e({\bf i}) = e(s_k{\bf i}) \psi_k $
\item $ (\psi_k \psi_l e({\bf i}))=(-1)^{p({\bf i}_k) p({\bf i}_{k+1}) p({\bf i}_{l}) p({\bf i}_{l+1})}(\psi_l \psi_k e({\bf i})) $ if $ |k-l|>1 $
\item $ (\psi_k e({\bf i}))(y_l e({\bf i})) = (y_l e(s_k {\bf i})) (\psi_k e({\bf i})) $ if $ |k-l|>1$
\item $ y_k y_l = y_l y_k $
\item $ \psi_k y_k e({\bf i}) = y_{k+1} \psi_k e({\bf i}) $ if $ {\bf i}_k \neq {\bf i}_{k+1} $
\item $ y_k \psi_k e({\bf i}) = \psi_k y_{k+1}  e({\bf i}) $ if $ {\bf i}_k \neq {\bf i}_{k+1} $
\item $ (y_k \psi_k - \psi_k y_{k+1}) e({\bf i}) = e({\bf i}) $ if $ {\bf i}_k = {\bf i}_{k+1} $
\item $ (\psi_k y_k - y_{k+1} \psi_k) e({\bf i}) = e({\bf i}) $ if $ {\bf i}_k = {\bf i}_{k+1} $
\item $ \psi_k^2 e({\bf i}) = 0 $ if $ {\bf i}_k = {\bf i}_{k+1} $
\item $ \psi_k^2 e({\bf i}) = y_{k+1} e({\bf i}) $ if $ {\bf i}_k=m, {\bf i}_{k+1}=m \pm 1 $
\item $ \psi_k^2 e({\bf i}) = y_{k} e({\bf i}) $ if $ {\bf i}_{k+1}=m, {\bf i}_k = m \pm 1 $
\item $ \psi_k^2 e({\bf i}) = (y_{k}+y_{k+1}) e({\bf i}) $ if $ |{\bf i}_{k} \bullet {\bf i}_{k+1}|=1 $ and ${\bf i}_k, {\bf i}_{k+1} \neq m $
\item $ \psi_k^2 e({\bf i}) = e({\bf i}) $ if $ |{\bf i}_k - {\bf i}_{k+1}|>1 $
\item $ (\psi_k \psi_{k+1} \psi_k - \psi_{k+1} \psi_k \psi_{k+1}) e({\bf i}) = \delta_{{\bf i}_k, {\bf i}_{k+2}} \delta_{|{\bf i}_k \bullet {\bf i}_{k+1}|,1} (1-\delta_{{\bf i}_k,m}) e({\bf i}) $. 
\end{enumerate}

There is a map $ d \colon R(\nu) \rightarrow R(\nu) $ which gives $ R(\nu) $ the structure of a differential graded algebra.  
On the generators set
\begin{equation*}
d(e({\bf i}))=0, \hspace{.5in}
d(y_k)=0, \hspace{.5in}
d(\psi_k)=\sum_{\substack {{\bf i} \\ i_k=i_{k+1}=m}} e({\bf i}).
\end{equation*}
Now extend $ d $ to the entire algebra in the obvious way.
$ R(\nu) $ is actually a bigraded algebra such that one of the gradings gives $ R(\nu) $ the structure of a dg algebra.  We present
this bigrading in Section ~\ref{graphicalpresentation}.

\subsection{Graphical presentation of $ R(\nu)$}
\label{graphicalpresentation}
There is a graphical presentation of $ R(\nu) $ very similar to that in ~\cite{KL1}.
We consider collections of smooth arcs in the plane connecting $ m $ points on one horizontal line with $ m $ points on another horizontal line.
Arcs are assumed to have no critical points (in other words no cups or caps) and each arc has a label from the set $ \lbrace 1, \ldots, m+n-1 \rbrace$.  
Arcs are allowed to intersect, but no triple intersections are allowed.
Arcs not labeled by $ m $ are allowed to carry dots and are called bosonic.  
Arcs labeled by $m$ will be referred to as fermionic.
Two diagrams that are related by an isotopy that does not change the combinatorial types of the diagrams or the relative position of crossings are taken to be equal.
The elements of the vector space $ R(\nu) $ are formal linear combinations of these diagrams modulo the local relations given below.
We give $ R(\nu) $ the structure of an algebra by concatenating diagrams vertically as long as the labels of the endpoints match.  If they do not, the product of two diagrams is taken to be zero.


\begin{equation}
\label{dotslide}
\begin{tikzpicture}
\draw (1,0) -- (1,2)[][thick];
\draw (2,0) -- (2,2)[][thick];
\draw (1.5,1) node{$\cdots$};
\filldraw [black] (1, .5) circle (2pt);
\filldraw [black] (2, 1.5) circle (2pt);
\draw (1,-.5) node{$i$};
\draw (2,-.5) node{$j$};
\draw (2.5,1) node{=};
\draw (3,0) -- (3,2) [][thick];
\draw (4,0) -- (4,2)[][thick];
\filldraw[black](4,.5) circle (2pt);
\filldraw [black] (3, 1.5) circle (2pt);
\draw (3.5,1) node{$\cdots$};
\draw (3,-.5) node{$i$};
\draw (4,-.5) node{$j$};

\draw (6,1) node{if $ i,j \neq m$};

\end{tikzpicture}
\end{equation}

\begin{equation}
\label{crossslide}
\begin{tikzpicture}[>=stealth]
\draw (0,0) .. controls (1,1) .. (1,2)[][thick];
\draw (1,0) .. controls (0,1) .. (0,2)[][thick];
\draw (2,0) .. controls (2,1) .. (3,2)[][thick];
\draw (3,0) .. controls (3,1) .. (2,2)[][thick];
\draw (1.5,1) node{$ \cdots$};
\draw (0,-.5) node{$i$};
\draw (1,-.5) node{$j$};
\draw (2,-.5) node{$k$};
\draw (3,-.5) node{$l$};
\draw (5,1) node{$ = (-1)^{p(i)p(j)p(k)p(l)}$};
\draw (7,0) .. controls (7,1) .. (8,2)[][thick];
\draw (8,0) .. controls (8,1) .. (7,2) [][thick];
\draw (10,0) .. controls (9,1) .. (9,2) [][thick];
\draw (9,0) .. controls (10,1) .. (10,2) [][thick];
\draw (8.5,1) node{$ \cdots$};
\draw (7,-.5) node{$i$};
\draw (8,-.5) node{$j$};
\draw (9,-.5) node{$k$};
\draw (10,-.5) node{$l$};
\end{tikzpicture}
\end{equation}

\begin{equation}
\label{fardotcrossslide}
\begin{tikzpicture}
\draw (0,0) .. controls (1,1) .. (1,2)[][thick];
\draw (1,0) .. controls (0,1) .. (0,2)[][thick];
\draw (2,0) -- (2,2)[][thick];
\draw (1.5,1) node{$\cdots$};
\filldraw [black] (2, 1.5) circle (2pt);
\draw (0,-.5) node{$i$};
\draw (1,-.5) node{$j$};
\draw (2,-.5) node{$k$};
\draw (2.5,1) node{=};
\draw (3,0) .. controls (4,1) .. (4,2)[][thick];
\draw (4,0) .. controls (3,1) .. (3,2)[][thick];
\draw (5,0) -- (5,2)[][thick];
\filldraw[black](5,.5) circle (2pt);
\draw (4.5,1) node{$\cdots$};
\draw (3,-.5) node{$i$};
\draw (4,-.5) node{$j$};
\draw (5,-.5) node{$k$};

\draw (7,0) -- (7,2)[][thick];
\draw (8,0) .. controls (9,1) .. (9,2)[][thick];
\draw (9,0) .. controls (8,1) .. (8,2)[][thick];

\filldraw[black](7,1.5) circle (2pt);
\draw (7.5,1) node{$\cdots$};
\draw (9.5,1) node{=};
\draw (10,0) -- (10,2)[][thick];
\draw (11,0) .. controls (11,1) .. (12,2)[][thick];
\draw (12,0) .. controls (12,1) .. (11,2)[][thick];
\filldraw[black](10,.5) circle (2pt);
\draw (10.5,1) node{$\cdots$};
\draw (7,-.5) node{$k$};
\draw (8,-.5) node{$j$};
\draw (9,-.5) node{$i$};
\draw (10,-.5) node{$k$};
\draw (11,-.5) node{$j$};
\draw (12,-.5) node{$i$};
\draw (13.5,1) node{if $ k \neq m$};
\end{tikzpicture}
\end{equation}

\begin{equation}
\label{crossdot1}
\begin{tikzpicture}[>=stealth]
\draw (0,0) -- (1,1)[][thick];
\draw (1,0) -- (0,1)[][thick];
\filldraw [black] (.25,.25) circle (2pt);
\draw (1.5,.5) node{=};
\draw (2,0) -- (3,1)[][thick];
\draw (3,0) -- (2,1)[][thick];
\draw (0,-.5) node {$i$};
\draw (1,-.5) node {$j$};
\draw (2,-.5) node{$ i$};
\draw (3,-.5) node{$j$};
\filldraw [black] (2.75, .75) circle (2pt);
\draw (8,.5) node{if $ i \neq m $ and $ i \neq j$};
\end{tikzpicture}
\end{equation}

\begin{equation}
\label{crossdot2}
\begin{tikzpicture}[>=stealth]
\draw (0,0) -- (1,1)[][thick];
\draw (1,0) -- (0,1)[][thick];
\draw (1.5,.5) node{=};
\draw (2,0) -- (3,1)[][thick];
\draw (3,0) -- (2,1)[][thick];
\draw (0,-.5) node {$i$};
\draw (1,-.5) node {$j$};
\draw (2,-.5) node{$ i$};
\draw (3,-.5) node{$j$};
\filldraw[black](.25,.75) circle (2pt);
\filldraw[black](2.75,.25) circle (2pt);
\draw (8,.5) node{if $ j \neq m $ and $ i \neq j$};
\end{tikzpicture}
\end{equation}

\begin{equation}
\label{nilhecke1}
\begin{tikzpicture}[>=stealth]
\draw (0,0) -- (1,1)[][thick];
\draw (1,0) -- (0,1)[][thick];
\filldraw [black] (.25,.25) circle (2pt);
\draw (1.5,.5) node{$-$};
\draw (2,0) -- (3,1)[][thick];
\draw (3,0) -- (2,1)[][thick];
\draw (0,-.5) node {$i$};
\draw (1,-.5) node {$i$};
\draw (2,-.5) node{$ i$};
\draw (3,-.5) node{$i$};
\filldraw [black] (2.75, .75) circle (2pt);
\draw (3.5,.5) node{=};
\draw (4,0) -- (4,1)[][thick];
\draw (5,0) -- (5,1)[][thick];
\draw (4,-.5) node {$i$};
\draw (5,-.5) node {$i$};
\draw (8.5,.5) node{if $ i \neq m $};
\end{tikzpicture}
\end{equation}

\begin{equation}
\label{nilhecke2}
\begin{tikzpicture}[>=stealth]
\draw (0,0) -- (1,1)[][thick];
\draw (1,0) -- (0,1)[][thick];
\filldraw [black] (.25,.75) circle (2pt);
\draw (1.5,.5) node{$-$};
\draw (2,0) -- (3,1)[][thick];
\draw (3,0) -- (2,1)[][thick];
\draw (0,-.5) node {$i$};
\draw (1,-.5) node {$i$};
\draw (2,-.5) node{$ i$};
\draw (3,-.5) node{$i$};
\filldraw [black] (2.75, .25) circle (2pt);
\draw (3.5,.5) node{=};
\draw (4,0) -- (4,1)[][thick];
\draw (5,0) -- (5,1)[][thick];
\draw (4,-.5) node {$i$};
\draw (5,-.5) node {$i$};
\draw (8.5,.5) node{if $ i \neq m $};
\end{tikzpicture}
\end{equation}

\begin{equation}
\label{nilhecke3}
\begin{tikzpicture}
\draw (0,0) .. controls (1,1) .. (0,2)[][thick];
\draw (1,0) .. controls (0,1) .. (1,2)[][thick];
\draw (0,-.5) node{$i$};
\draw (1,-.5) node{$i$};
\draw (1.5,1) node{$=$};
\draw (2,1) node{$0$};
\end{tikzpicture}
\end{equation}

\begin{equation}
\label{mm+1}
\begin{tikzpicture}
\draw (0,0) .. controls (1,1) .. (0,2)[][thick];
\draw (1,0) .. controls (0,1) .. (1,2)[][thick];
\draw (0,-.5) node{$m$};
\draw (1,-.5) node{$m\pm1$};
\draw (1.5,1) node{$=$};
\draw (2,0) -- (2,2)[thick];
\draw (3,0) -- (3,2)[thick];
\draw (2,-.5) node{$m$};
\draw (3,-.5) node{$m\pm1$};
\filldraw[black](3,1) circle (2pt);
\draw (5,0) .. controls (6,1) .. (5,2)[][thick];
\draw (6,0) .. controls (5,1) .. (6,2)[][thick];
\draw (5,-.5) node{$m \pm1$};
\draw (6,-.5) node{$m$};
\draw (6.5,1) node{$=$};
\draw (7,0) -- (7,2)[thick];
\draw (8,0) -- (8,2)[thick];
\draw (7,-.5) node{$m \pm1$};
\draw (8,-.5) node{$m$};
\filldraw[black](7,1) circle (2pt);
\end{tikzpicture}
\end{equation}

\begin{equation}
\label{Rnu1}
\begin{tikzpicture}
\draw (0,0) .. controls (1,1) .. (0,2)[][thick];
\draw (1,0) .. controls (0,1) .. (1,2)[][thick];
\draw (0,-.5) node{$i$};
\draw (1,-.5) node{$j$};
\draw (1.5,1) node{$=$};
\draw (2,0) -- (2,2)[thick];
\draw (3,0) -- (3,2)[thick];
\draw (2,-.5) node{$i$};
\draw (3,-.5) node{$j$};
\draw (3.5,1) node{$+$};
\draw (4,0) -- (4,2)[thick];
\draw (5,0) -- (5,2)[thick];
\draw (4,-.5) node{$i$};
\draw (5,-.5) node{$j$};
\filldraw[black](2,1) circle (2pt);
\filldraw[black](5,1) circle (2pt);
\draw (8,1) node{if $ |i \bullet j| = 1$ and $ i,j \neq m$};
\end{tikzpicture}
\end{equation}

\begin{equation}
\label{Rnu2}
\begin{tikzpicture}
\draw (0,0) .. controls (1,1) .. (0,2)[][thick];
\draw (1,0) .. controls (0,1) .. (1,2)[][thick];
\draw (0,-.5) node{$i$};
\draw (1,-.5) node{$j$};
\draw (1.5,1) node{$=$};
\draw (2,0) -- (2,2)[thick];
\draw (3,0) -- (3,2)[thick];
\draw (2,-.5) node{$i$};
\draw (3,-.5) node{$j$};
\draw (5,1) node{if $ i \bullet j = 0$};
\end{tikzpicture}
\end{equation}

\begin{equation}
\label{reid3}
\begin{tikzpicture}[>=stealth]
\draw (0,0) -- (2,2)[][thick];\draw (2,0) -- (0,2)[][thick];
\draw (1,0) .. controls (0,1) .. (1,2)[][thick];
\draw (2.5,1) node {=};
\draw (3,0) -- (5,2)[][thick];
\draw (5,0) -- (3,2)[][thick];
\draw (4,0) .. controls (5,1) .. (4,2)[][thick];
\draw (0,-.5) node {$i$};
\draw (1,-.5) node {$j$};
\draw (2,-.5) node {$k$};
\draw (3,-.5) node {$i$};
\draw (4,-.5) node {$j$};
\draw (5,-.5) node {$k$};
\draw (9,1) node {Unless $ i=k$ and $ |i \bullet j| = 1$};
\end{tikzpicture}
\end{equation}

\begin{equation}
\label{reid3nu}
\begin{tikzpicture}[>=stealth]
\draw (0,0) -- (2,2)[][thick];\draw (2,0) -- (0,2)[][thick];
\draw (1,0) .. controls (0,1) .. (1,2)[][thick];
\draw (2.5,1) node {$-$};
\draw (3,0) -- (5,2)[][thick];
\draw (5,0) -- (3,2)[][thick];
\draw (4,0) .. controls (5,1) .. (4,2)[][thick];
\draw (0,-.5) node {$i$};
\draw (1,-.5) node {$j$};
\draw (2,-.5) node {$i$};
\draw (3,-.5) node {$i$};
\draw (4,-.5) node {$j$};
\draw (5,-.5) node {$i$};
\draw (5.5,1) node{=};
\draw (11,1) node {if $ |i \bullet j| = 1$};
\draw (6.5,1) node {$(1-\delta_{i,m})$};
\draw (7.5,0) -- (7.5,2)[][thick];
\draw (8.5,0) -- (8.5,2)[][thick];
\draw (9.5,0) -- (9.5,2)[][thick];
\draw (7.5,-.5) node {$i$};
\draw (8.5,-.5) node {$j$};
\draw (9.5,-.5) node {$i$};
\end{tikzpicture}
\end{equation}
This algebra is naturally bigraded.  The degrees of the generators are given in ~\eqref{degrees}, where $ i \bullet j $ was defined in ~\eqref{bullets}.

\begin{equation}
\label{degrees}
\begin{tikzpicture}
\draw (-1,0) -- (7,0)[][very thick];
\draw (0,.5) node{bidegree};
\draw (0, -1) node{generator};
\draw (1,1) -- (1,-2)[][very thick];
\draw (1.5,.5) node{$(0,0)$};
\draw (1.5,-.5) -- (1.5,-1.5)[thick];
\draw (1.5, -1.75) node{$i$};
\draw (2,1) -- (2, -2)[very thick];
\draw (3,-.5) -- (3,-1.5)[thick];
\draw (3, -1.75) node{$ i$};
\draw (4,1) -- (4, -2)[very thick];
\draw (3, .5) node{$(0,i \bullet i)$};
\filldraw[black](3,-1) circle (2pt);
\draw (5,-1.5) -- (6,-.5)[thick];
\draw (6, -1.5) -- (5, -.5)[thick];
\draw (5,-1.75) node{$i$};
\draw (6, -1.75) node{$j$};
\draw (5.5,.5) node{$(-\delta_{i,m} \delta_{j,m}, -i \bullet j)$};
\draw (-1,1) -- (-1,-2)[very thick];
\draw (7,1) -- (7,-2)[very thick];
\draw (-1,1) -- (7,1)[][very thick];
\draw (-1,-2) -- (7,-2)[][very thick];

\end{tikzpicture}
\end{equation}

As with negative dg gradual algebras, if $ M $ is a graded dg left $ R(\nu) $-module, denote by $ M^{a,b} $ the subspace in degree $ (a,b)$.  Then define the shifted module $ M \{ r \} \lbrace s \rbrace $ by $ (M [r] \lbrace s \rbrace)^{a,b} = M^{a+r,b-s}$ where the action of $ R(\nu) $ and the differential $d$ are twisted as explained in Section ~\ref{dgbasics}.
Unless specified otherwise we will assume from now on that a dg module is a graded dg module.

As in ~\cite[Section 2.1]{KL1}, there is a grading-preserving dg-module anti-involution $ \sigma \colon R(\nu) \rightarrow R(\nu) $ given by flipping a diagram about a horizontal axis.  If $ M $ is a left (respectively right) 
$ R(\nu)$-module,
then we let $ M^{\sigma} $ be the right, (respectively left) $ R(\nu)$-module with the same underlying vector space with the action of $ R(\nu) $ now twisted by $ \sigma$.

For a sequence of the form $ {\bf i}=(i_1^{(n_1)}, \ldots, i_r^{(n_r)}) $ such that $ n_1 i_1 + \cdots + n_r i_r = \nu $ and $n_k=1$ when $i_k=m$ we will define an idempotent $ \widetilde{e}({\bf i}) $.
First, for an element $ w \in S_n $ let $ w=w_1 \cdots w_p $ be a minimal presentation and define $ \psi_w = \psi_{w_1} \cdots \psi_{w_p} $.  Note that $ \psi_w $ in general does depend upon the choice of the minimal presentation.
Let $ w_{0, {\bf i}} $ be a minimal representative of the longest element of $ S_{n_1} \times \cdots \times S_{n_r} $.  
Note that $ w_{0, {\bf i}} $ does not depend on the presentation of this longest element.
Finally set
\begin{equation*}
\widetilde{e}({\bf i}) = \psi_{w_{0, {\bf i}}} (x_1^{n_1-1} \cdots x_{n_1-1}) \cdots (x_{n_1+\cdots+n_{r-1}+1}^{n_r-1} \cdots x_{n_1+\cdots+n_{r}-1}) e(\underbrace{i_1, \ldots, i_1}_{n_1}, \ldots, \underbrace{i_r, \ldots, i_r}_{n_r})
\end{equation*}
This element is an idempotent and $ d(\widetilde{e}({\bf i}))=0$.  Notice that if $ n_1= \cdots = n_r=1 $ then $ \widetilde{e}({\bf i})=e(i_1, \ldots, i_r) $.
Now we define the dg indecomposable summand
\begin{equation*}
P_{i_1^{(n_1)}, \ldots, i_r^{(n_r)}} = R(\nu) \widetilde{e}({\bf i}) \Big\{ \sum_{k=1}^r \frac{-n_k(n_k-1)}{2} \Big\}.
\end{equation*}

\subsection{The polynomial representation $ Pol_{\nu}$.}
Let $ t = \sum_{i, i\neq m} \nu_i \in \mathbb{N}$.
Define the vector space 
$$ Pol_{\nu} = \bigoplus_{{\bf i} \in \textrm{Seq}(\nu), w \in S_{\nu_m}} Pol({\bf i},w), $$
where
$$ Pol({\bf i},w) = \Bbbk[x_1({\bf i},w), \ldots, x_t({\bf i},w)]. $$

Let $ g \in Pol({\bf i},w)$.
Let $ 1 \leq k \leq t$.  Suppose the $ k$-th bosonic entry (counting from the left) of $ {\bf i} $ occurs in the $ \bar{k}$-th position of $ {\bf i}$, (counting from the left).
Let $ s_k g $ be the polynomial obtained from $ g $ by replacing each generator $ x_r(\bf i, w) $ by $ x_{s_k r}(s_{\bar k} \bf i, w)$.
Let $ \widehat{s}_{k} g $ be the polynomial obtained from $ g $ by replacing each generator $ x_r(\bf i, w) $ by $ x_r(s_{\bar k} \bf i, w)$.


We now define an action of $ R(\nu) $ on $ Pol_{\nu}$.  Boson-boson or boson-fermion crossings will change the subscript of each generator of $ Pol_{\nu} $ along with the first entry inside the parenthesis.
Fermion-fermion crossings will change only the second argument in the parenthesis of each generator.

\begin{itemize}
\item Let $ e({\bf i}) $ act as the identity on $ Pol({\bf j},w) $ if $ {\bf i} = {\bf j} $ and acts by zero otherwise.
\item A diagram of vertical lines for a sequence $ {\bf i} $ with a dot on the $ k$-th bosonic strand, counting from the left, takes $ g \in Pol({\bf i}, w) $ to $ x_k({\bf i},w)g$.

\item A crossing of the $ k$-th and $(k+1)$-st  bosonic strands, assuming that they are next to each and of the same color map $ g \in Pol({\bf i},w) $ to the divided difference 
$\partial_k(g) = \frac{g - s_k g}{x_k({\bf i}, w)-x_{k+1}({\bf i}, w)}  \in Pol({\bf i},w)$.

\item A crossing of the $ k$-th and $(k+1)$-st bosonic strands, assuming that they are next to each other, and the color of the $k$th bosonic strand is one less than that of the $(k+1)$-st bosonic strand, maps
$ g \in Pol({\bf i},w) $ to $ (x_k(s_{\bar k}{\bf i},w)+x_{k+1}(s_{\bar k}{\bf i},w))(s_k g) \in Pol(s_{\bar k}{\bf i},w)$.

\item A crossing of the $ k$-th and $(k+1)$-st bosonic strands, assuming that they are next to each other, and the color of the $k$th bosonic strand is one more than that of the $(k+1)$-st bosonic strand, maps
$ g \in Pol({\bf i},w) $ to $ s_k g \in Pol(s_{\bar k}{\bf i},w)$.

\item A crossing of the $ k$-th and $(k+1)$-st bosonic strands, assuming that they are next to each other, and the color of the $k$-th bosonic strand differs by more than one of that of the $(k+1)$-st bosonic strand, maps
$ g \in Pol({\bf i},w) $ to $ s_k g \in Pol(s_{\bar k}{\bf i},w)$.  

\item A crossing of the $ k$-th and $(k+1)$-st fermionic strands, assuming that they are next to each other, maps
$ g \in Pol({\bf i},w) $ to $ \rho_w^k g' \in Pol({\bf i},s_k w)$ if $ l(s_k w)=l(w)+1 $ where $ g' $ is obtained from $ g $ by replacing each variable $ x_j({\bf i},w) $ with $ x_j({\bf i}, s_k w)$ and $ \rho_w^k $ is determined by
$ \sigma_{s_k w} =  \rho_w^k \sigma_{s_k} \sigma_{s_w} $ in the LOT algebra.
If $ l(s_k w) = l(w)-1$, then $ g $ gets mapped to zero.

\item A crossing of the $ k$-th bosonic and $ j$-th fermionic strands, assuming they are next to each other, the bosonic strand is on the main diagonal, and the color of the bosonic strand is $ m+1$,  maps $ g \in Pol({\bf i},w) $ to 
$ \widehat{s}_k g \in Pol(s_{\bar k}{\bf i},w)$.

\item A crossing of the $ k$-th bosonic and $ j$-th fermionic strands, assuming they are next to each other, the bosonic strand is on the main diagonal, and the color of the bosonic strand is $ m-1$,  maps $ g \in Pol({\bf i},w) $ to 
$  \widehat{s}_k g \in Pol(s_{\bar k} {\bf i},w)$.

\item A crossing of the $ k$-th bosonic and $ j$-th fermionic strands, assuming they are next to each other, the bosonic strand is on the main diagonal, and the color of the bosonic strand is not $ m-1$ or $m+1$,  maps $ g \in Pol({\bf i},w) $ to 
$ \widehat{s}_k g \in Pol(s_{\bar k}{\bf i},w)$.

\item A crossing of the $ k$-th bosonic and $ j$-th fermionic strands, assuming they are next to each other, the fermionic strand is on the main diagonal, and the color of the bosonic strand is $ m+1$,  maps $ g \in Pol({\bf i},w) $ to 
$ x_k(s_{{\bar k}-1}{\bf i},w) \widehat{s}_{{\bar k}-1} g \in  Pol(s_{{\bar k}-1}{\bf i},w)$.

\item A crossing of the $ k$-th bosonic and $ j$-th fermionic strands, assuming they are next to each other, the fermionic strand is on the main diagonal, and the color of the bosonic strand is $ m-1$,  maps $ g \in Pol({\bf i},w) $ to 
$ x_k(s_{\bar{k}}{\bf i},w) \widehat{s}_{{\bar k}-1} g \in  Pol(s_{{\bar k}-1}{\bf i},w)$.

\item A crossing of the $ k$-th bosonic and $ j$-th fermionic strands, assuming they are next to each other, the fermionic strand is on the main diagonal, and the color of the bosonic strand is not $ m-1$ or $m+1$,  maps $ g \in Pol({\bf i},w) $ to 
$ \widehat{s}_{{\bar k}-1} g \in  Pol(s_{{\bar k}-1}{\bf i},w)$.
\end{itemize}

\begin{prop}
The above action gives $ Pol_{\nu} $ the structure of an $ R(\nu)$-module.
\end{prop}

\begin{proof}
We only sketch a verification of the relations of $ R(\nu) $ given in Section ~\ref{R(nu)} and leave the details to the reader.
\begin{itemize}
\item Equations 1-4 are trivial to verify.
\item Equation 5 is easy to check and is a relation checked for the polynomial representation in \cite{Kh1}.
\item Equation 6 is easy and is contained in \cite{KL1}.
\item Equation 7 is trivial.
\item Equations 8 and 9 are checked case by case.  Most of the cases are contained in ~\cite{Kh1} or ~\cite{KL1}.  The cases where $ i_k = m-1$ and $ i_{k+1}=m $ or where $ i_k=m$ and $ i_{k+1}=m-1$ are trivial to verify.
\item Equations 10-12 are all proved in ~\cite{KL1}.
\item Equations 13 and 14 are contained in \cite{Kh1}.
A representative case that needs to be checked for relation 14 is when the $ k$th bosonic strand is labeled by $ m-1 $ and is immediately to the left of the $ j$th fermionic strand.
Then $ {\bf i}_{\bar k}=m-1 $ and $ {\bf i}_{\bar k +1}=m$.
Then $ \psi_{\bar k} e({\bf i}) $ acts on the generator $ x_r(\bf i, w) $ of $ Pol({\bf i},w) $ by mapping it to $ x_k(s_{\bar k}{\bf i},w) x_r(s_{\bar k}{\bf i}, w)$.  Then
$ \psi_{\bar k} e(s_{\bar k} {\bf i}) $ maps this generator to $ x_k({\bf i}, w) x_r({\bf i}, w)$. 
\item Equations 15 and 16 are checked in ~\cite{KL1}.
\item Equation 17 requires a case by case analysis.
\end{itemize}
\end{proof}

Let $ {\bf i}, {\bf j} \in Seq(\nu)$ and $ d = |\nu|$.  We will construct a basis of $ e({\bf j}) R(\nu) e({\bf i}) $ very similar to the basis constructed in ~\cite[Section 2.3]{KL1}.
Let $ {}_{\bf j} S_{\bf i} $ be the subset of $ S_d $ consisting of permutations that send $ {\bf i} $ to $ {\bf j} $ via the obvious action of permutations on sequences.  
For each $ w \in {}_{\bf j} S_{\bf i} $, we convert its minimal presentation $ \widetilde{w} $ into an element of $ e({\bf j}) R(\nu) e({\bf i}) $.  Denote it by $ \widehat{w}_{\bf i}$.  
Let
$$ {}_{\bf j} \widehat{S}_{\bf i} = \lbrace \widehat{w}_{\bf i} | w \in {}_{\bf j} S_{\bf i} \rbrace \subset e({\bf j}) R(\nu) e({\bf i}). $$
See ~\cite[Example 2.4]{KL1}.
As mentioned in ~\cite{KL1}, $ {}_{\bf j} \widehat{S}_{\bf i} $ depends on choices of minimal presentations.  
Finally, let
$$ {}_{\bf j} B_{\bf i} = \lbrace \psi y_1^{u_1} \cdots y_d^{u_d} e({\bf i}) | \psi \in {}_{\bf j} \widehat{S}_{\bf i}, u_i \in \Z_{\geq 0} \rbrace. $$

\begin{prop}
\label{basis}
\begin{enumerate}
\item $ e({\bf j}) R(\nu) e({\bf i}) $ is a free graded abelian group with a homogeneous basis  $ {}_{\bf j} B_{\bf i} $.
\item The action of $ R(\nu) $ on $ Pol_{\nu} $ is faithful.
\end{enumerate}
\end{prop}

\begin{proof}
The proof is nearly identical to the proof of \cite[Theorem 2.5]{KL1}.
In the base case of the induction given there, one needs to note that in addition to the action of the nil-Hecke algebra on the polynomial representation being faithful, the action of the LOT algebra (see ~\cite{LOT1} for the original definition or ~\cite{Kh1} for an exposition), on itself is also faithful.
In the induction step, the map $ \bar{\delta} $ (see \cite[Theorem 2.5]{KL1} for its definition) needs to be modified slightly if the $ l$-th strand from the left on the bottom is labeled by $ m$ and the strands which terminate in positions $ k $ and $ k+1 $ on the top intersect.
In that case, one needs to again remove the crossing causing that intersection and now put a dot on the bottom of the strand terminating at the point $ k+1$ from the left, on the top.  In ~\cite{KL1} the dot was always placed on the bottom of the strand terminating at the point $ k $ from the left, on the top (after the intersection was removed).
\end{proof}

\subsection{Characters and the Shuffle lemma}
\label{characters}

Let $ M $ be a finite-dimensional bigraded left $ R(\nu)$-module.
Let $ M^{a,b} $ be the subspace in bidegree $ (a,b)$.
Set $ \text{gdim} M(q,t) = \sum_{(a,b)} (\text{dim} M^{a,b}) t^b q^a$.
Define 
$$ ch_{q,t}(M) = \sum_{{\bf i}} \text{gdim} (e({\bf i})M) {\bf i}. $$
Clearly, the character of such an object is a formal linear combination of sequences $ {\bf i} $ with coefficients in $ \mathbb{Z}[q^{\pm 1}, t^{\pm 1}]$.

The proof of the Shuffle Lemma stated below follows from \cite[Section 2.6]{KL1}.  For the analogue of the proof of \cite[Proposition 2.16]{KL1}, which is a crucial step in the proof of the Shuffle Lemma, it is important to work in the category of bigraded modules and ignore the dg structure.  Otherwise, filtrations rather than direct sums are obtained. 

\begin{lemma}
\label{shuffle}
$$ ch_{q,t}(\Ind_{R(\nu) \otimes R(\nu')}^{R(\nu + \nu')} (M \otimes N)) = ch_{q,t}(M) \shuffle ch_{q,t}(N)$$.
\end{lemma}



\section{Main theorem} 
\label{maintheorem}
In this section we consider the case $n=1$ so that $I''$ is empty and prove that $ {\bf f}(m,1) $ is isomorphic as a twisted bialgebra to the Grothendieck group of the derived category of compact graded dg modules over the sum $ \oplus_{\nu} R(\nu) $.

It is important to take $n=1$ so that Proposition ~\ref{R(nu)isngdg} is true.  For other values of $n$ the algebra $ R(\nu) $ would be non-trivial for every degree $ j \in \mathbb{Z} $ of the second grading.

\begin{prop}
\label{R(nu)isngdg}
$ R(\nu) $ is a negative dg gradual algebra.
\end{prop}

\begin{proof}
The first condition is obvious since by definition $R(\nu)$ is supported in non-positive cohomological degrees.  

One could easily exhibit a finite generating set consisting of diagrams with fermionic crossings only which gives the third condition.

It is left to show that $ R(\nu)^{0,*} $ is a gradual algebra.  As in the case of ~\cite{KL1}, the center is isomorphic to a tensor product of rings of symmetric polynomials in $ \nu_i$ variables for $i=1, \ldots, m-1 $
and thus is Noetherian.
The proof of ~\cite[Corollary 2.10]{KL1} implies that $ R(\nu)^{0,*} $ is finitely generated over its center.

\end{proof}

\subsection{Kato modules}
We will need the so-called {\it Kato modules} $ L(i^k)$.
Let $ L(i) $ denote the one-dimensional simple graded dg $R(i)$-module which lives in degree $ (0,0)$.
For $ i \neq m$, set 
$$ L(i^k) = \Ind_{R(i) \otimes \cdots \otimes R(i)}^{R(ki)} L(i) \otimes \cdots \otimes L(i).$$

For $ k>1$, let $ L(m^k) $ be the irreducible dg $ R(km)$-module with basis $ \lbrace w_0, w_{-1} \rbrace $ 
where the cohomological degree of $w_i$ is $i$,
such that $ \psi_r e({\bf i}) w_0 = w_{-1} $,  $ \psi_r e({\bf i}) w_{-1}=0  $ for $ r = 1, \ldots, k-1 $.  In other words,
an $ (m,m) $ crossing maps $ w_{-1} $ to $ w_0 $ and maps $ w_0 $ to zero.
Set $ d(w_{-1})=w_0 $ and $ d(w_0)=0$.  It is trivial to check that this gives $ L(m^k) $ the structure of an irreducible dg $ R(km)$-module.

\subsection{The character map}
For the rest of this subsection we work in the category of finite-dimensional graded dg left modules.  By a module, we will mean an object in this category.
For a finite-dimensional, graded dg $ R(\nu)$-module $ M$, define the specialized character:
$$ ch(M) = \sum_{{\bf i}} \text{gdim} (e({\bf i})M)(q, -1){\bf i}. $$

Thus the character of a finite-dimensional module is a formal linear combination of sequences $ {\bf i} $ with coefficients in $ \mathbb{Z}[q,q^{-1}]$.  
Clearly $ ch(M) $ is the evaluation of $ ch_{q,t}(M) $ at $ t=-1$.
Note that in ~\cite{KL1}, $ ch(M)$ was in the set $ \mathbb{N}[q,q^{-1}]$.

Let $ \mathcal{L}_{\nu} $ denote the set of isomorphism classes of simple dg $ R(\nu)$-modules of type I.
Note that $d$ acts by zero on any such module.

\begin{prop}
\label{chwelldefined}
Let $ L_1 $ and $ L_2 $ be two simple objects which are homotopically equivalent as dg $ R(\nu)$-modules.  Then $ ch(L_1) = ch(L_2)$.
\end{prop}

\begin{proof}
First assume that $ L $ is homotopically trivial.
Since $ L $ is a graded dg module, it is a complex of singly graded modules:
$$ 0 \rightarrow \bigoplus_a L^{-r,a} \rightarrow \cdots \rightarrow \bigoplus_a L^{0,a} \rightarrow 0. $$
Since the differential of this complex respects the $ q$-grading and left multiplication by elements of the form $ e({\bf i})$, and $ L $ is homotopically trivial,
each complex of vector spaces
$$ 0 \rightarrow e({\bf i}) L^{-r,a} \rightarrow \cdots \rightarrow e({\bf i}) L^{0,a} \rightarrow 0 $$
is homotopic to zero.  
Thus 
$ \sum_{b=0}^r (-1)^b \text{dim}(e({\bf i}) L^{-b,a}) =0$
for each $ {\bf i} $ and $ a$.
Thus $ ch(L) = 0$.

Now suppose $ L_1 \simeq L_2$.  Then in the abelian category of complexes of modules, there is an exact sequence
$$ 0 \rightarrow L_2 \rightarrow L \rightarrow L_1[1] \rightarrow 0 $$
where $ L $ is the cone of the map from $ L_1 $ to $ L_2$.  Since $ L \simeq 0$, 
$ ch(L_1) = ch(L_2)$.
\end{proof}

We now follow the exposition in ~\cite[Section 3.2]{KL1} in order to prove that the character map is injective on $ \mathbb{Z}[q,q^{-1}]\mathcal{L}_{\nu}$.  The proofs of the propositions are similar to those in ~\cite{KL1} which are similar to the arguments of ~\cite[Chapter 5]{Klesh}.

Let $ M $ be a (dg) $ R(\nu)$-module.  Define an $ R(\nu-ki) \otimes R(k i)$-module $ {\Res}_{R(\nu)}^{R(\nu-ki) \otimes R(k i)} M $ by 
$$ \text{Res}_{R(\nu)}^{R(\nu-ki) \otimes R(k i)} M = (e(\nu - ki) \otimes e(ki)) M $$
where $ e(\mu) = \sum_{{\bf j} \in Seq(\mu)} e({\bf j})$.

\begin{lemma}
\label{adjointness}
There is a natural isomorphism in the abelian (and in the homotopy) category of differentially graded modules: 
$$ \text{Hom}_{R(\nu)}(\text{Ind}_{R(\nu-ki) \otimes R(ki)}^{R(\nu)}(N \otimes L(i^k)),M)
\cong 
\text{Hom}_{R(\nu-ki) \otimes R(ki)}(N \otimes L(i^k), \text{Res}_{R(\nu)}^{R(\nu-ki) \otimes R(k i)} M).$$
\end{lemma}

\begin{proof}
See ~\cite[Section 10.11]{BL}).
\end{proof}

\begin{defined}
For a graded dg module $ M $ let $ \text{soc}(M) $ be the sum of all simple graded dg submodules of $ M$ and let $ \text{hd}(M)=M/J(M) $, where $J(M) $ is the intersection of all maximal graded dg submodules of $M$.
Then define functors $e_{ki}, f_{ki}$ and operators $ \widetilde{e}_{ki}, \widetilde{f}_{ki}$ by:
\begin{enumerate}
\item $ e_{ki} M = \text{Res}_{R(\nu)}^{R(\nu-ki) \otimes R(k i)} M $ viewed as a $ R(\nu -ki) $-module,
\item $ f_{ki} M = \text{Ind}_{R(\nu) \otimes R(ki)}^{R(\nu+ki)}(M \otimes L(i^k)) $,
\item $ \widetilde{e}_{ki}(M) = \text{soc}(e_{ki}M) $,
\item $ \widetilde{f}_{ki}(M) = \text{hd} ({f}_{ki} M)$.
\end{enumerate}
\end{defined}

The proof of the next lemma is trivial.
\begin{lemma}
\label{resch}
$ ch(e_{ki} M) = \sum_{{\bf j} \in Seq(\nu-ki)} \text{gdim}(e({\bf j}i^k)M)(q,-1){\bf j} $.
\end{lemma}

If $ k>1$, then $ ch(e_{km} M) =0$.
Let $ \epsilon_i(M) = \text{max} \lbrace k \geq 0 | e_{ki} M \neq 0 \rbrace$.

\begin{lemma}
\label{consecutiveentries}
Let $L$ be a simple dg module of type I.  
If $ e({\bf i}) L \neq 0 $, then there cannot be consecutive entries of $ m $ in $ {\bf i}$.
\end{lemma}

\begin{proof}
Suppose 
$ {\bf i} = (i_1, \ldots, i_{r-1}, m,m, i_{r+2}, \ldots, i_d)$
and $ e({\bf i}) v = v$ for some $ v \in L$.  Then by the dg structure, $ \psi_r v \neq 0 $ and so the simple module $ L $ lives in at least two different homological degrees. This means 
that $ L $ is a type II simple module and thus zero in the Grothendieck group.  Therefore  $ \epsilon_m(L) \leq 1$.
\end{proof}

\begin{lemma}
Suppose $ M $ is an irreducible dg $ R(\nu)$-module of type I and $ N \otimes L(i^k)\lbrace r \rbrace [s] $ is an irreducible submodule of 
$ e_{ki} M $ for some $ 0 \leq k \leq \epsilon_i(M) $ and $ r,s \in \Z$.  
Then $ \epsilon_i(N) = \epsilon_i(M)-k$.
\end{lemma}

\begin{proof}
The argument is the same as in ~\cite[Lemma 3.6]{KL1} which faithfully follows that of ~\cite[Lemma 5.1.2]{Klesh}.
The proof uses Lemma ~\ref{adjointness} and the Shuffle Lemma.
\end{proof}

\begin{lemma}
Suppose $ N $ is an irreducible dg $ R(\nu)$-module of type I and $ \epsilon_i(N)=0$.  Let $ M = \text{Ind}_{R(\nu) \otimes R(ki)}^{R(\nu+ki)}(N \otimes L(i^k))$ where $k=1$ if $i=m$. Then
\begin{enumerate}
\item $ \text{Res}_{R(\nu+ki)}^{R(\nu) \otimes R(k i)} M \cong N \otimes L(i^k)$,
\item $ \text{hd} M $ is irreducible and $ \epsilon_i(\text{hd}M)=k$,
\item all other composition factors $L$ of $ M$ have $ \epsilon_i(L) < k$.
\end{enumerate}
\end{lemma}

\begin{proof}
This is ~\cite[Lemma 5.1.3]{Klesh} or ~\cite[Lemma 3.7]{KL1} in the graded case.
\end{proof}

\begin{lemma}
\label{restrictionlemma}
Let $ M $ be an irreducible dg $ R(\nu)$-module of type I and $ k = \epsilon_i(M)$.  Then 
$ \text{Res}_{R(\nu)}^{R(\nu-ki) \otimes R(k i)} M \cong N \otimes L(i^k) $ for
some irreducible $ R(\nu-ki) $-module $ N $ such that $ \epsilon_i(N)=0$.
\end{lemma}

\begin{proof}
This is ~\cite[Lemma 3.8]{KL1} whose proof is the same as ~\cite[Lemma 5.1.4]{Klesh}.
\end{proof}

\begin{lemma}
Suppose $ i \neq m $ and 
let $ \mu = (\mu_1, \ldots, \mu_s) $ be a composition of $ k$.
\begin{enumerate}
\item The module $ L(i^k) $ over $ R(ki) $ is the only graded irreducible dg module up to isomorphism and grading shifts.
\item All composition factors of $ \text{Res}_{R(ki)}^{R(\mu_1 i) \otimes \cdots \otimes R(\mu_s i)} L(i^k) $ are isomorphic to
$ L(i^{\mu_1}) \otimes \cdots \otimes L(i^{\mu_s}) $, up to grading shifts, and $ \text{soc}(\text{Res}_{R(ki)}^{R(\mu_1 i) \otimes \cdots \otimes R(\mu_s i)} L(i^k)) $ is irreducible.
\end{enumerate}
\end{lemma}

\begin{proof}
This is ~\cite[Proposition 3.11]{KL1} whose proof follows that in ~\cite{Klesh}.
\end{proof}

\begin{lemma}
\begin{enumerate}
\item An irreducible dg module over $ R(km) $ is isomorphic to $ L(m^k) $ up to a shift in grading.
\item The dg module $ \text{Res}_{R(km)}^{R(\mu_1 m) \otimes \cdots \otimes R(\mu_s m)} L(m^k) $ has an irreducible head and socle.
\end{enumerate}
\end{lemma}

\begin{proof}
The case $ k=1$ is trivial since $ R(m) $ is just the ground field.
Assume that $ k>1$.
By Proposition ~\ref{simpledgmodules} each irreducible module $L$ is either concentrated in degree $ (a,b) $ for some $ a $ and $ b $ or degrees $ (a,b) $ and $ (a-1,b) $.
Note that we are using the fact that the $q$-grading is trivial for $ R(km)$.
Suppose that $ L$ is concentrated in one degree and let $ w$ be a non-zero element in this degree.  Then $ \psi_r e({\bf i}) w =0$.  However,
$ d(\psi_r e({\bf i})w)=w $ which implies that $ w=0$.  Thus $ L $ is concentrated in degrees $ (a,b) $ and $ (a-1,b) $.  

Let $ w $ be an element in degree $ (a,b) $.  Since $ R(km) $ is negatively graded in the first degree and $ d(\psi_r e({\bf i})w)=w $, if the dimension of $ L $ in degree
$ (a,b) $ is greater than one, any element in this degree generates a proper submodule.  Thus we may assume that dimension of this graded subspace is one and spanned by $ w$.
Then by Proposition ~\ref{simpledgmodules}, $ L \cong L(m^k)$.

The dg module $ \text{Res}_{R(km)}^{R(\mu_1 m) \otimes \cdots \otimes R(\mu_s m)} L(m^k) $ is actually irreducible if at least one of the $ \mu_i >1 $ because then there is a non-trivial action of $d$ which maps
$ w_{-1} $ to $w_0$.  If all of the $ \mu_i=1$ then $w_0$ spans the irreducible socle. The quotient of $ \text{Res}_{R(km)}^{R(\mu_1 m) \otimes \cdots \otimes R(\mu_s m)} L(m^k) $ by the socle is the irreducible head.
\end{proof}

\begin{lemma}
\label{firreduciblelemma}
Let $ N $ be an irreducible dg $ R(\nu)$-module of type I and 
$ M = \text{Ind}_{R(\nu) \otimes R(ki)}^{R(\nu+ki)} N \otimes L(i^k)$ where $k=1$ if $i=m$.
Then $ \text{hd}M $ is irreducible, $ \epsilon_i(\text{hd}M) = \epsilon_i(N)+k $, and all other composition factors $ L $ of  $ M $ have $ \epsilon_i(L) < \epsilon_i(N)+k$.
\end{lemma}

\begin{proof}
This is ~\cite[Lemma 3.9]{KL1} whose proof is identical to that of ~\cite[Lemma 5.1.5]{Klesh}.  Note that the proof relies on the fact that the module induced from finite tensor products of $ L(i)$ is simple,
which is true for $ i \neq m$ and for $i=m$, $ \text{Ind}_{R(m) \otimes R(m)}^{R(2m)} (L(m) \otimes L(m)) $ is irreducible.
\end{proof}

\begin{lemma}
\label{formofsocle}
Assume $ 0 \leq k \leq \epsilon_i(M)$, for some $ i \neq m$, or if $i=m$ assume $ 0 \leq k \leq \epsilon_m(M)=1$.
Then for any irreducible dg $ R(\nu) $-module $ M $ of type I,
$ \soc(\Res_{R(\nu)}^{R(\nu-ki) \otimes R(k i)} M) $ is an irreducible $ R(\nu-ki) \otimes R(ki)$-module of the form
$ L \otimes L(i^k) $ with 
$ \epsilon_i(L)=\epsilon_i(M)-k$.
\end{lemma}

\begin{proof}
This is ~\cite[Proposition 3.10]{KL1} whose proof is the same as ~\cite[Theorem 5.1.6]{Klesh}.  The proof depends upon the form of the socle of the restricted Kato module $ L(i^k)$ which is different in the case $ i=m$.   Note that if $ k =1$, this follows trivially from Lemma ~\ref{restrictionlemma}.
\end{proof}

\begin{lemma}
\label{eirreduciblelemma}
Let $ M $ be an irreducible dg $ R(\nu)$-module of type I with $ \epsilon_i(M) >0$ for some $ i \neq m $ or $ \epsilon_i(M)=1$ for $i=m$.  Then $ \widetilde{e}_i(M) $ is irreducible and $ \epsilon_i(\widetilde{e}_iM) = \epsilon_i(M)-1$.
If $ i \neq j$, then $ \widetilde{e}_iM \not\cong \widetilde{e}_j M$.
\end{lemma}

\begin{proof}
This is ~\cite[Corollary 3.12]{KL1} which is a graded version
of ~\cite[Corollaries 5.1.7, 5.1.8]{Klesh}.  
Note that the proof relies on Lemma ~\ref{formofsocle} which gives the conditions on 
$ \epsilon_i(M)$.
\end{proof}

\begin{lemma}
\label{irrcrystaloperators}
Let $ M $ be an irreducible dg $ R(\nu)$-module of type I.  
\begin{enumerate}
\item If $ i \neq m$ then $ \widetilde{f}_i M $ is an irreducible dg module of type I. 
\item If $i=m$ then $\widetilde{f}_i M $ is an irreducible dg module of type I if $ \epsilon_m(M)=0 $ and of type II if $ \epsilon_m(M) = 1$.
\item $ \widetilde{e}_i M $ is an irreducible dg module of type I or zero. 
\end{enumerate}
\end{lemma}

\begin{proof}
The proof of the first and second parts follow from Lemma ~\ref{firreduciblelemma} and the third part follows from Lemma ~\ref{eirreduciblelemma}.
\end{proof}

\begin{lemma}
Let $ M $ be an irreducible dg module of type I.  
\begin{enumerate}
\item $ \epsilon_i(M) = \text{max} \lbrace k \geq 0 | \widetilde{e}_i^k M \neq 0 \rbrace$.
\item If $ i \neq m $ or $i=m$ and $ \epsilon_i(M)=0$, then $ \epsilon_i(\widetilde{f}_iM) = \epsilon_i(M)+1$.
\end{enumerate}
\end{lemma}

\begin{proof}
The first part is just a restatement of Lemma ~\ref{eirreduciblelemma}.
The second part follows easily from Lemma ~\ref{firreduciblelemma}.  
In the second part if $ \epsilon_m(M)=1$ then $ \widetilde{f}_m(M)$ is an irreducible dg module of type II.
\end{proof}

\begin{lemma}
Let $ M $ be an irreducible dg $ R(\nu)$-module of type I.  Then
\begin{enumerate}
\item $ \text{soc}(e_{ki} M) \cong (\widetilde{e}_i^k M) \otimes L(i^k)$, 
\item $ \text{hd}(f_{ki} M) \cong \widetilde{f}_i^kM$ if $ i \neq m $,
\item $ \text{hd}(f_{m} M) \cong \widetilde{f}_m M$
\end{enumerate}
\end{lemma}

\begin{proof}
This is ~\cite[Lemma 3.13]{KL1} whose proof is the same as ~\cite[Lemma 5.2.1]{Klesh}.  
The last part is true by definition.  We do not consider $  \text{hd}(f_{km} M) $ for $k>1$ since a type II module is produced. 
\end{proof}



\begin{lemma}
\label{crystal}
Let $ M $ be an irreducible dg $ R(\nu)$-module of type I and $ N$ an irreducible dg $ R(\nu+i)$-module of type I.  
\begin{enumerate}
\item If $ i \neq m$, then $ \widetilde{f}_i M \cong N $ if and only if
$ \widetilde{e}_i N \cong M$.
\item Suppose $ \epsilon_m(M)=0$.  Then 
$ \widetilde{f}_m M \cong N $ if and only if
$ \widetilde{e}_m N \cong M$.
\end{enumerate}
\end{lemma}

\begin{proof}
This is ~\cite[Lemma 3.15]{KL1} whose proof is the same as ~\cite[Lemma 5.2.3]{Klesh}.  Note that the proof of the first part does not work when $ i=m $, except in the easy case stated in the second part.
\end{proof}

\begin{lemma}
\label{crystalcor}
Let $ M $ and $ N $ be irreducible dg $ R(\nu)$-modules of type I.  
\begin{enumerate}
\item Suppose $ i \neq m$.
Then $ \widetilde{f}_i M \cong \widetilde{f}_i N $ if and only if $ M \cong N$.
If $ \epsilon_i(M), \epsilon_i(N) > 0$, then $ \widetilde{e}_i M \cong \widetilde{e}_i N $ if and only $ M \cong N$.
\item Suppose $ \epsilon_m(M) = \epsilon_m(N) = 0$.  Then 
$ \widetilde{f}_m M \cong \widetilde{f}_m N $ if and only if $ M \cong N$.
Suppose $ \epsilon_m(M) = \epsilon_m(N) = 1$.  Then 
$ \widetilde{e}_m M \cong \widetilde{e}_m N $ if and only if $ M \cong N$.
\end{enumerate}
\end{lemma}

\begin{proof}
This is ~\cite[Corollary 3.16]{KL1} or ~\cite[Corollary 5.2.4]{Klesh} whose proof follows from Lemma ~\ref{crystal}.
\end{proof}

\begin{prop}
\label{chinjective}
The character map $ ch \colon G_0(R(\nu)) \rightarrow \Z[q,q^{-1}]\Seq(\nu) $ is injective.
\end{prop}

\begin{proof}
This follows as in the proof of ~\cite[Theorem 5.3.1]{Klesh}.  We repeat the arguments because there are some small modifications when operators $ \widetilde{e}_m $ and $ \widetilde{f}_m $ are needed.

We proceed by induction on $ d = |\nu|$.  The case $ d=1 $ is trivial so suppose $ d > 1 $ and there is a linear dependence relation $ \sum c_L ch(L)=0$ over $ \Z[q,q^{-1}] $.
We will show by downward induction on $ k = d, \ldots, 1 $ that $ c_L = 0 $ for all $ L $ of type I with $ \epsilon_i(L)=k$ for some $i$.
Since every irreducible $ L $ has $ \epsilon_i(L) > 0 $ for at least one $ i \in I $, this is enough.
Note that for $L$ of type I, $ \epsilon_m(L) \leq 1$ by Lemma ~\ref{consecutiveentries}.


Consider the case $ k=d$.  Then $ \Res_{R(\nu)}^{R(di)} L = 0 $ except if $ L \cong L(i^d)$, (which is non-trivial).  Thus applying $ \Res_{R(\nu)}^{R(di)} $ to the dependence equation proves that the coefficient of
$ L(i^d) $ must be zero.  This is the base case of the second induction.
Now suppose that $ 1 \leq k < d $ and we have shown that $ c_L = 0 $ for all $ L $ with $ \epsilon_i(L)>k$.  
Apply $ \Res_{R(\nu)}^{R(\nu-ki) \otimes R(ki)} $ to the dependence equation to get:
$$ \sum_{\epsilon_i(L)=k} c_L ch(\Res_{R(\nu)}^{R(\nu-ki) \otimes R(ki)} L)=0. $$

Now each such $ \Res_{R(\nu)}^{R(\nu-ki) \otimes R(ki)} L $ is isomorphic to $ \widetilde{e}_i^k L \otimes L(i^k) $.  By Lemma ~\ref{crystalcor}, if $ L $ is not congruent to $ L' $, then
$ \widetilde{e}_i^kL $ is not congruent to $ \widetilde{e}_i^kL'$.  Then the induction hypothesis on $ d $ gives that all the coefficients are zero.
\end{proof}

\subsection{Categorical bilinear forms}
Let $ \nu = \sum_{i \in I} n_i i$ and $ (\nu)_q = \prod_{i \neq m} \frac{1}{1-q^{i \bullet i}} $.

There is a $ \Z[q,q^{-1}] $-bilinear pairing
\begin{equation}
\label{K0G0}
( , )_{K_0,G_0} \colon K_0(R(\nu)) \times G_0(R(\nu)) \rightarrow \Z[q,q^{-1}]
\end{equation}
given by $ ([P],[M]) \mapsto \text{gdim} \RHom_{R(\nu)}(P,M)(q,-1) = \text{gdim}(P^{\sigma} \otimes_{R(\nu)} M)(q,-1)$.

There is also $ \Z[q,q^{-1}] $-bilinear pairing
\begin{equation}
\label{K0G0}
( , )_{K_0} \colon K_0(R(\nu)) \times K_0(R(\nu)) \rightarrow \Z[q,q^{-1}](\nu)_q
\end{equation}
given by $ ([P],[M]) \mapsto \text{gdim}(\RHom_{R(\nu)}(P,M))(q,-1)= \text{gdim}(P^{\sigma} \otimes_{R(\nu)} M)(q,-1)$.

\begin{prop}
\label{formK0properties}
The pairing $ ( , )_{K_0} $ (which we denote simply as $(,)$) has the following properties:
\begin{enumerate}
\item $ (\Bbbk, \Bbbk) = 1$,
\item $ ([P_i],[P_j]) = \frac{\delta_{i,j}}{1-q^{i \bullet i}} $ if $i \neq m$,
\item $ ([P_m],[P_m]) = 1$,
\item $ ([P], [\Ind_{R(\nu') \otimes R(\nu'')}^{R(\nu)} M \otimes N]) = ([\Res_{R(\nu)}^{R(\nu') \otimes R(\nu'')}P], [M \otimes N]) $,
\item $ ([\Ind_{R(\nu') \otimes R(\nu'')}^{R(\nu)} M \otimes N],[P]) = ([M \otimes N],[\Res_{R(\nu)}^{R(\nu') \otimes R(\nu'')} P]) $.
\end{enumerate}
\end{prop}

\begin{proof}
The proofs of the first, second, fourth, and fifth equations are the same as ~\cite[Proposition 3.3]{KL1} but now using adjointness in the dg category as in Lemma ~\ref{adjointness}.
The third equation is immediate since $P_m$ is a simple one-dimensional dg module.
\end{proof}

\subsection{Identification of the Grothendieck group}
Assembling the results of this section, we finally categorify $ \mathcal{U}_q^+(\mathfrak{gl}(m|1))$.

The direct sum of Grothendieck groups $ \oplus_{\nu} K_0(R(\nu)) $ is naturally an $\mathcal{A}$-algebra.  If $M$ and $N$ are graded dg $R(\nu')$ and $ R(\nu'')$-modules respectively then set
\begin{equation*}
[M]\cdot[N]=[\Ind_{R(\nu') \otimes R(\nu'')}^{R(\nu'+\nu'')} M \otimes N].
\end{equation*}
It's clear that $ [P_{i_1 \cdots i_k}] \cdot [P_{i_{k+1} \cdots i_l}] = [P_{i_1 \cdots i_l}] $.

Restriction functors endow $ \oplus_{\nu} K_0(R(\nu)) $ with the structure of a coalgebra:    
\begin{equation*}
[\Res] \colon K_0(\bigoplus_{\nu}R(\nu)) \rightarrow K_0(\bigoplus_{\nu}R(\nu)) \otimes K_0(\bigoplus_{\nu}R(\nu))
\end{equation*}
where if $ M$ is a graded dg $R(\nu)$-module.  Then 
\begin{equation*}
[\Res M] = [\bigoplus_{\substack {\nu',\nu'' \\ \nu'+\nu''=\nu}} \Res_{R(\nu)}^{R(\nu') \otimes R(\nu'')} M].
\end{equation*}

\begin{prop}
Let  $ {\bf k} \in \textrm{Seq}(\nu'+\nu'')$.  Then $ \Res_{R(\nu'+\nu'')}^{R(\nu')\otimes R(\nu'')} P_{{\bf k}} $ has a filtration whose subquotients are isomorphic to
$ P_{\bf i} \otimes P_{\bf j}[deg_1(D({\bf i}, {\bf j}, {\bf k}))] \langle deg_2(D({\bf i}, {\bf j}, {\bf k})) \rangle  $ where $ {\bf i} \in \textrm{Seq}(\nu')$, $ {\bf j} \in \textrm{Seq}(\nu'')$ such that $ {\bf k} $ could be obtained from a shuffle of $ {\bf i} $ and $ {\bf j}$ and $ D({\bf i}, {\bf j}, {\bf k}) $ is the diagram representing the shuffle.
\end{prop}

\begin{proof}
This is similar to ~\cite[Proposition 2.19]{KL1} where restriction takes a projective module to a direct sum of projective modules.  Here it is a filtration because the direct sum is not preserved by the action of $d$.
In this filtration the object corresponding to the trivial shuffle is a dg submodule.
\end{proof}

\begin{prop}
\label{homomorphism}
There exists a non-zero homomorphism of bialgebras
$ \gamma \colon {\bf f} \rightarrow \bigoplus_{\nu \in \N[I]} \mathbb{Q}(q) \otimes_{\mathbf{Z}[q,q^{-1}]} K_0(R(\nu)) $
where 
$$ \gamma(\theta_{i_1} \cdots \theta_{i_k}) = [P_{i_1 \ldots i_k}]. $$
\end{prop}

\begin{proof}
It suffices to check that the generators of $ \mathcal{J} $ get mapped to zero by $ \gamma $.
Since we assume that $ I'' $ is empty, the first generator of $ \mathcal{J} $ given in Definition ~\ref{idealJ} does not exist.

The dg structure of $ R(\nu) $ gives that $ \gamma(\theta_m^2) = 0 $ since the corresponding projective object is homotopically trivial as in ~\cite{Kh1}.
The fact that $ \gamma(\theta_i \theta_j - \theta_j \theta_i) = 0 $ for $ |i-j|>1$ follows from the invertibility of the crossing generators with those labels.  See ~\cite[Proposition 3.4]{KL1} for more details.
Finally, let $ \zeta = (q+q^{-1})\theta_i \theta_{i \pm 1} \theta_i - \theta_i^2 \theta_{i \pm 1} - \theta_{i \pm 1} \theta_i^2 $.  For $ i<m-1$ and $ i=m-1 $ in the minus case, $ \gamma(\zeta)=0 $ by ~\cite[Proposition 3.4]{KL1}. 
For $ i=m-1 $ in the plus case, $ \gamma(\zeta)=0 $ following the arguments of ~\cite{Kh1} which relies on the proof of ~\cite[Proposition 3.4]{KL1} again. 
\end{proof}

The map in Proposition ~\ref{homomorphism} restricts to a homomorphism of $ \mathcal{A}$-algebras
\begin{equation*}
\gamma \colon {}_{\mathcal{A}} \mathbf{f} \rightarrow \bigoplus_{\nu \in \N[I]} K_0(R(\nu))
\end{equation*}
where
\begin{equation*}
\gamma(\theta_{i_1}^{(n_1)} \cdots \theta_{i_r}^{(n_r)}) = [P_{i_1^{(n_1)}, \ldots,  i_r^{(n_r)}} ].
\end{equation*}






\begin{prop}
\label{injective}
The homomorphism $ \gamma  \colon {}_{\mathcal{A}} {\bf f} \rightarrow \bigoplus_{\nu \in \N[I]} 
K_0(R(\nu)) $ is injective.
\end{prop}

\begin{proof}
Since $ (\gamma(x),\gamma(y))_{K_0} = (x,y)$ and $ ( , ) $ is non-degenerate on $ {}_{\mathcal{A}}{\bf f}$, if $ \gamma(x)=0$, then $ x=0$.
\end{proof} 

\begin{theorem}
\label{bijective}
The homomorphism $ \gamma \colon {\bf f} \rightarrow \bigoplus_{\nu \in \N[I]} \mathbb{Q}(q) \otimes_{\mathbb{Z}[q,q^{-1}]} K_0(R(\nu)) $
is an isomorphism.
\end{theorem}

\begin{proof}
This is similar to the argument in \cite[Section 3.2]{KL1}.
Injectivity of $ \gamma $ is Proposition ~\ref{injective}.  By dualizing the character map, Proposition ~\ref{chinjective} implies that $ \gamma $ is surjective.
\end{proof}

\begin{corollary}
There is an isomorphism of $\mathcal{A}$-coalgebras $ \gamma \colon {}_{\mathcal{A}} \mathbf{f} \rightarrow \bigoplus_{\nu \in \N[I]} K_0(R(\nu)) $.
\end{corollary}

\begin{proof}
This follows from Theorem ~\ref{bijective} and the argument in ~\cite{KL1}.
\end{proof}










\begin{remark}
An interesting problem is to construct a categorification of the upper half of quantum $ \mathfrak{gl}(m|n) $ or even more immediately the upper half of quantum $ \mathfrak{gl}(2|2) $.  The case $ n \geq 2 $ seems more difficult than the case 
$ n=1 $ because of the presence of the Serre relation
\begin{equation*}
\label{4term}
(q+q^{-1})\theta_{m} \theta_{m-1} \theta_{m+1} \theta_m =
\theta_{m-1} \theta_{m} \theta_{m+1} \theta_{m} 
+ \theta_{m} \theta_{m-1} \theta_{m} \theta_{m+1}
+ \theta_{m} \theta_{m+1} \theta_{m} \theta_{m-1}
+ \theta_{m+1} \theta_{m} \theta_{m-1} \theta_{m}.
\end{equation*}
The dg algebras introduced here would need to be enhanced in order to categorify this relation.
\end{remark}









\section{References} 


\def\refname{}

\end{document}